\documentclass[twocolumn]{autart}    

\usepackage{amssymb}                                                       
\usepackage{ifpdf}
\usepackage{amsmath}
\usepackage{amsfonts}
\usepackage{mathrsfs}
\usepackage{color}
\usepackage{graphicx}
\usepackage{subfigure}
\usepackage{epstopdf}
\usepackage{lineno}
\usepackage{cite}
\usepackage{enumerate}
\usepackage{algorithm}
\usepackage{algpseudocode}
\usepackage{bm}

\newcommand{\order}[1]{\mathcal{O}\left(#1\right)}

\newcommand{\prt}[1]{\left(#1\right)}
\newcommand{\brk}[1]{\left[#1\right]}
\newcommand{\crk}[1]{\left\{#1\right\}}

\newcommand{\norm}[1]{\left\Vert #1 \right\Vert}

\newtheorem{theorem}{Theorem}

\newtheorem{corollary}{Corollary}

\newtheorem{lemma}{Lemma}     
\newtheorem{remark}{Remark}
\newtheorem{proposition}{Proposition}
\newtheorem{assumption}{Assumption}

\newcommand{\1}{\mathbf{1}}
\newcommand{\T}{\intercal}

\definecolor{cuhkpl}{RGB}{152,24,147}
\definecolor{revisiongray}{gray}{0.55}

\allowdisplaybreaks

\begin{document}

\begin{frontmatter}

\title{Distributed risk-averse optimization via CVaR\thanksref{footnoteinfo}} 

\thanks[footnoteinfo]{This work was supported by the Swedish Research Council Distinguished Professor Grant 2017-01078, Knut and Alice Wallenberg Foundation, Wallenberg Scholar Grant, and Swedish Strategic Research Foundation SUCCESS Grant FUS21-0026.}

\author[DCS]{Siyi Wang}\ead{siyiw@kth.se},    
\author[DCS]{Kun Huang}\ead{kunhuang@kth.se},               
\author[DCS]{Lei Xu}\ead{lei5@kth.se},  
\author[DCS]{Karl H. Johansson}\ead{kallej@kth.se}

\address[DCS]{Department of Decision and Control Systems, School of Electrical Engineering and Computer Science, KTH Royal Institute of Technology,
10044  Stockholm, Sweden}


\begin{abstract}
Distributed systems often operate under uncertainty, where minimizing expected loss may overlook rare but severe events.
This paper studies a distributed risk-averse convex optimization problem in which agents cooperatively minimize the average of local conditional value-at-risk (CVaR) objectives over a time-varying network. 
Each agent has access only to noisy evaluations of its local loss function, 
rather than to its CVaR objective or gradient.
We therefore develop a zeroth-order algorithm that uses 
sampled losses to construct empirical CVaR estimates and their gradient estimates.
At each iteration, agents combine neighboring decisions and perform a local update. 
Under convexity and Lipschitz continuity assumptions, 
we prove that the agents reach exact asymptotic consensus. 
We also establish a finite-time expected suboptimality bound for the weighted ergodic iterate.  
With diminishing step sizes and fixed sample sizes, the local last iterates converge almost surely to a common optimum, and their limiting expected CVaR gap is bounded in terms of the smoothing and finite-sample errors.  
This distributed bound matches the parameter dependence of the centralized benchmark provided in this paper. 
Finally, simulations on a distributed sensor network estimation problem illustrate the efficacy of the method.
\end{abstract}

\end{frontmatter}

\section{Introduction}
\label{sec:0_intro}

Distributed optimization is a fundamental framework for decision-making in large-scale networked systems~\cite{nedic2018distributed}. In this framework, agents cooperatively minimize a global objective using local computation and information exchanged with their neighbors. Applications include wireless sensor network localization~\cite{simonetto2014distributed}, distributed estimation~\cite{rabbat2004distributed}, and multi-robot systems~\cite{shorinwa2024distributeda,shorinwa2024distributedb}.
In stochastic environments,  distributed optimization methods often
handle uncertainty by minimizing the average expected loss $J(x) = \sum_{i=1}^m\mathbb{E}_{\xi^i}[ J_i(x;\xi^i)]/m$ \cite{sundhar2010distributed,xie2025differentially}. 
Recent research has developed a broad range of distributed optimization methods under different problem and network settings. 
Examples include decentralized stochastic gradient methods with changing topologies and local 
updates~\cite{koloskova2020unified}, stochastic gradient tracking~\cite{pu2021distributed}, 
asynchronous schemes~\cite{even2024asynchronous}, 
and gradient-free algorithms based on local function 
evaluations~\cite{pang2022gradient,yi2022zeroth,tang2023zeroth,yuan2024distributed}. 
Convergence guarantees have been established for various objectives under different assumptions on local information and communication networks.

In high-stakes applications, such as autonomous navigation and distributed power dispatch, 
rare but severe events can have catastrophic consequences. 
For instance, a sensor anomaly may lead to fatal collisions in autonomous navigation, 
while a localized power surge can cascade into a grid-wide blackout \cite{basaran2026comprehensive}. 
Therefore, optimizing only the expected cost may leave these systems vulnerable to tail-risk events.
This motivates risk-averse distributed optimization that explicitly accounts for tail risk.
Among various risk measures, conditional value at risk (CVaR) has attracted significant attention due to its coherence and convexity.
CVaR captures tail risk by measuring the expected loss above a prescribed quantile. 
In this paper, we study a distributed stochastic convex optimization problem over networks, where agents collaboratively optimize a shared CVaR objective with respect to a global decision variable.
Specifically, consider a network of agents $ [m] := {1,\ldots,m} $ that collaboratively solve the optimization problem
\begin{equation}\label{eq problem}
\min_{x\in\mathcal{X}} \mathcal{C}(x)= 
\frac{1}{m}\sum_{i=1}^{m} C^i(x),
\end{equation}
where $x$ denotes the decision variable, $\mathcal{X}$ is the feasible set, and $C^i(\cdot)$ denotes the local CVaR objective induced by the
stochastic loss $J_i(x,\xi^i)$ of agent $i$.
 
CVaR has been studied in several multi-agent settings, 
including game theory \cite{wang2022risk,wang2024learning}, reinforcement learning \cite{al2024risk,ishikawa2025decentralized}, and multi-armed bandits \cite{tan2023cvar}. 
These works typically optimize individual utilities \cite{wang2022risk, wang2024learning}, policies, or sequential decisions rather than a shared stochastic convex objective over a communication network \cite{tan2023cvar}. 
Distributed risk-aware optimization has also been studied through distributionally robust formulations that minimize worst-case risk over an ambiguity set~\cite{jiao2022distributed,lan2023optimal,ning2026collaborative}. 
In contrast, the CVaR formulation considered here targets tail losses and admits a tractable implementation. 
The main challenge is that an agent generally knows neither the distribution of its local uncertainty nor an analytical expression for its CVaR objective, and hence cannot directly evaluate a CVaR gradient. 
To the best of our knowledge, no existing work has addressed distributed stochastic convex optimization of CVaR objectives over networks.

In this paper, we develop a zeroth-order distributed risk-averse learning algorithm for minimizing the average of local CVaR functions over time-varying networks.
At each iteration, every agent mixes the decisions received from its neighbors, perturbs the mixed point in a random direction, and queries multiple noisy values of its local loss. The resulting empirical CVaR value is used to construct a one-point estimate of the gradient of a smoothed local CVaR objective. The agent then performs a projected update.
Our analysis separates three sources of error: network disagreement, zeroth-order smoothing, and finite-sample CVaR estimation. Under standard convexity and Lipschitz assumptions and uniformly jointly connected, doubly stochastic communication graphs, we first derive a finite-time expected suboptimality bound for the weighted ergodic average.
For polynomially decaying step sizes, the limiting ergodic error is of order $\order{\delta_{\max} + e_s/\delta_{\min}}$.
Here, $\delta_{\max}$ and $\delta_{\min}$ are the largest and smallest smoothing radii, respectively, and $e_s$ quantifies the sampling accuracy.
We further prove exact asymptotic consensus. When each agent uses a fixed sample size, its last iterate converges almost surely to a common optimizer of a fixed convex surrogate, and the limiting expected CVaR gap is  $\order{\delta_{\max}+e_s/\delta_{\min}}$.
A centralized counterpart removes the network-disagreement term. Its upper bound has similar dependence on the horizon, smoothing radii, and sampling accuracy, but it does not contain the graph-dependent constants present in the distributed upper bound.

The main contributions of this paper are summarized as follows: 
\begin{enumerate}
    \item   We formulate a distributed risk-averse stochastic convex optimization problem with CVaR as the risk measure and develop a zeroth-order distributed algorithm for solving it. 
    \item Under standard convexity, Lipschitz continuity, and diminishing step size assumptions, we prove exact asymptotic consensus and characterize ergodic convergence. For fixed per-agent sample sizes, we further prove ordinary last-iterate convergence and bound the limiting expected CVaR gap by  $\order{\delta_{\max}+e_s/\delta_{\min}}$.
    \item We develop and analyze a centralized zeroth-order benchmark. The centralized and distributed upper bounds exhibit similar dependence on the iteration horizon, smoothing radii, and sampling accuracy.  
\end{enumerate}

The remainder of this paper is organized as follows. 
Section~\ref{sec:preliminary} introduces the CVaR and communication-graph preliminaries.
Section~\ref{sec:main result} develops the zeroth-order distributed risk-averse optimization algorithm. 
Section~\ref{sec:analysis} analyzes its convergence.
Section~\ref{sec:centralize} introduces and analyzes a centralized benchmark for comparison.
Section~\ref{sec:simulation} evaluates the proposed algorithm through numerical simulations.
Section~\ref{sec:5_conc} concludes the paper, 
and Section~\ref{sec:appendix} provides auxiliary results and complementary proofs.

\textbf{Notation:} The symbol $\|\cdot\|$ denotes the Euclidean norm. The notation $\mathcal O(\cdot)$ suppresses multiplicative constants, whereas $\widetilde{\mathcal O}(\cdot)$ additionally suppresses logarithmic factors.  
For two subsets $A$ and $B$ of a Euclidean space, their Minkowski sum is $A\oplus B:=\{a+b\mid a\in A,\ b\in B\}$.

\section{Preliminaries}\label{sec:preliminary}
Consider $m$ agents indexed by $[m]:=\{1,\ldots,m\}$. Agent $i$ has a stochastic loss $J^i:\mathcal X\times\Xi_i\to\mathbb R$, where $x\in\mathcal X$ is the decision variable, $\mathcal X\subset\mathbb R^d$ is nonempty, compact, and convex, and $\xi^i\in\Xi_i$ is the local uncertainty. 
We denote the diameter of $\mathcal X$ by $D_x:=\sup_{x,y\in\mathcal X}\|x-y\|$.
Additionally, we assume that $\mathcal{X}\subseteq \mathbb{R}^d $ contains the ball of radius $r>0$ centered at the origin, i.e.,
 $
r\mathbb{B} \subseteq \mathcal{X}$,
 where 
$\mathbb{B}:=\{x\in\mathbb{R}^d:\|x\|_2\le 1\}.$

\subsection{CVaR} 
We use CVaR as the risk measure. 
Given a risk level $\alpha_i\in(0,1)$, 
we define the local CVaR objective of agent $i$ through the Rockafellar--Uryasev representation \cite{rockafellar2000optimization}:
\begin{equation*}
    \begin{aligned}
       C^i(x) & :={\rm{CVaR}}_{\alpha_i}\left[J^i(x,\xi^i)\right]  \\
     & = \min_{\tau\in\mathbb{R}}
     \left\{\tau+\frac{1}{\alpha_i}
     \mathbb{E}\left[\left(J^i(x,\xi^i)-\tau\right)_+\right]\right\},
    \end{aligned}
\end{equation*}
where $(z)_+:=\max\{z,0\}$ denotes the positive part of $z$.

\subsection{Graph}
Consider a time-varying undirected graph $\mathcal{G}_k = (\mathcal{V}, \mathcal{E}_k)$, where $\mathcal{V} = \{1, \dots, m\}$ denotes the set of agents and $\mathcal{E}_k$ denotes the set of edges at iteration $k$. An edge $i \leftrightarrow j \in \mathcal{E}_k$ indicates that agents $i$ and $j \neq i $ can exchange messages at iteration $k$.  
Each agent has access only to its local function evaluations, 
and the decision variables communicated by its neighbors.
Given $\mathcal{G}_k$, we let $\mathcal{N}_k^i$ denote the neighbor set of agent $i$ at iteration $k$: $\mathcal{N}_k^i  = \{j \in \mathcal{V} \setminus\{i\}  \mid i \leftrightarrow j \in \mathcal{E}_k\} \cup \{i\}$. The inclusion of agent $i$ itself reflects the fact that each agent has access to its own information.  
The time-varying graph model captures changes in information flow, such as routing variations caused by sensor mobility, that cannot be represented by a fixed topology.

We make the following assumptions on the graph $\mathcal{G}_k$.
\begin{assumption}\label{assumption:graph q}
There exists a scalar $q$ such that the graph $(\mathcal{V},\cup_{l=1,\dots,q}\mathcal{E}_{k+l})$ is connected for all $k$. 
\end{assumption}
The graph sequence $\{\mathcal{G}_k\}_{k\geq 0}$ specifies the
available communication links over time. 
To describe how agents combine information received through these links, we associate with
each graph $\mathcal{G}_k$ a weight matrix
$W_k=[w_k^{ij}]$.
We impose the following standard assumptions on the weight matrices
to ensure information mixing over the time-varying graph sequence. 
\begin{assumption}\label{assumption:graph}
For every $k \geq 0$, the matrix $W_k$ has the following properties:
\begin{enumerate}
    \item $W_k$ is doubly stochastic,i.e., $\1^{\T}W_k = \1^{\T}$ and $W_k \1 = \1 $.
    \item $W_k$ is compatible with the structure of the graph $\mathcal{G}_k$, i.e., $w_k^{ij}=0$ whenever $j\notin\mathcal N_k^i$.
    \item $W_k$ has positive diagonal entries, i.e., $w_k^{ii} > 0$ for all $i \in \mathcal{V}$.
    \item There is a $\varrho>0$ such that $w_k^{ij}\geq\varrho$ for all $i\in\mathcal V$ and $j\in\mathcal N_k^i$.
\end{enumerate}
\end{assumption} 

Using this information structure, the next section develops a distributed algorithm for minimizing the average CVaR objective in~\eqref{eq problem}.

\section{Distributed risk-averse algorithm}\label{sec:main result} 
This section develops the zeroth-order distributed risk-averse method summarized in Algorithm~\ref{alg:zeroth-order}.

Each agent uses sampled local losses to estimate its CVaR and then constructs a one-point gradient estimate for a smoothed local CVaR objective. 
We begin by smoothing CVaR for gradient estimation.
 Let $\mathbb{S}^{d-1} = \crk{u \in \mathbb{R}^d: \| u\| =1}$ denote the boundary of $\mathbb{B}^d$. We introduce a smoothed approximation of $C^i(x)$ following the techniques in \cite{flaxman2004online}:
\begin{equation}\label{eq:smoothed cvar}
    C_\delta^i(x) = \mathbb{E}_{\nu \sim {\rm Unif}(\mathbb{B})}[C^i(x+\delta_i \nu )].
\end{equation}   
Here, $\delta_i>0$ is the smoothing radius of agent $i$. We define $\delta_{\max}:=\max_{i\in\mathcal V}\delta_i$ and assume that $\delta_{\max}<r$.
It follows from \cite{duchi2015optimal} that the gradient of $C_\delta^i(x)$ can be expressed as
\begin{equation}\label{eq:two point gradient estimate}
    \nabla C_\delta^i(x) = \mathbb{E}_{u\sim {\rm Unif}(\mathbb{S}^{d-1})}\left[\frac{d}{\delta_i} C^i(x+\delta_i u ) u\right].
\end{equation}
This identity motivates an estimator of the smoothed
CVaR gradient. 
At iteration $k$, let $y_k^i$ denote the point at which agent $i$
estimates its local CVaR gradient, and let $\hat y_k^i = y_k^i+\delta_i u_k^i$ be its perturbed counterpart, where $u_k^i\sim {\rm Unif}(\mathbb{S}^{d-1})$.
Inspired by \eqref{eq:two point gradient estimate}, agent $i$ constructs
the CVaR gradient estimate
\begin{equation}\label{eq:estimated gradient}
    \hat{g}_k^i = \frac{d}{\delta_i}  \widehat{C}_k^i(\hat{y}_k^i)    u_k^i.
\end{equation} 
Here, 
$\widehat C_k^i(\hat y_k^i)$ denotes the empirical CVaR estimate
computed from the queried losses at $\hat{y}_k^i$. 
Because it uses an empirical rather than exact CVaR value, this estimator is biased with respect to the gradient in~\eqref{eq:smoothed cvar}.

We next describe how the empirical CVaR estimates in \eqref{eq:estimated gradient} are constructed. 
At each iteration $k$, agent $i$ first obtains the point $y_k^i$ through local
communication. Specifically, agent $i$ sends its local decision
$x_k^i$ to its neighbors and receives $x_k^j$ from agents
$j\in\mathcal{N}_k^i$. It then computes the weighted aggregation
\begin{equation}\label{eq:update y}
    y_k^i
    =
    \sum_{j\in\mathcal{N}_k^i} w_k^{ij}x_k^j .
\end{equation}
Given the perturbed point $\hat y_k^i$, agent $i$ queries its local function and constructs the empirical distribution function
\begin{align}\label{eq:EDF}
 \hat{P}_k^{i}(z,\hat{y}_k^i)&=\frac{1}{s_k^i} \sum_{j=1}^{s_k^i} \mathbf{1} \{J^i(\hat{y}_k^i, \xi_k^{i,j}) \leq z \}.
\end{align} 
Here, $J^i(y,\xi)$ denotes a noisy function evaluation at
the point $y$, and $s_k^i\geq 1$ denotes the number of queried
function evaluations at each point for agent $i$ at iteration $k$. 
For each agent $i$, all samples $\{\xi_k^{i,j}\}_{k,j}$ are drawn independently from the same fixed distribution $P_i$ 
and are independent of the perturbation directions and the past iterates.
The sample sizes are chosen to satisfy
\begin{align}\label{eq:sampling requirement}
    \sum_{i=1}^{m} \frac{1}{m\sqrt{s_k^i}} \le e_s, \quad \text{for~all}~ k,
\end{align}
where $e_s>0$ is a tuning scalar. 
Given the empirical distribution $\hat P_k^i$, 
we define the empirical
CVaR by the Rockafellar--Uryasev representation:
\begin{align*}
\widehat C_k^i(\hat y_k^i)
:=
\min_{\tau\in\mathbb R}
\left\{
\tau+
\frac{1}{\alpha_i s_k^i}
\sum_{j=1}^{s_k^i}
\left[
J^i(\hat y_k^i,\xi_k^{i,j})-\tau
\right]_+
\right\}.
\end{align*}

\begin{algorithm}[t]
\caption{Zeroth-order distributed risk-averse learning}
\label{alg:zeroth-order}
\begin{algorithmic}[1]
    \Require Initial values $x_0^i$ and $y_0^i$, iteration horizon $T$,
    smoothing parameters $\delta_i$, learning rates $\eta_k$, and risk
    levels $\alpha_i$

    \For{$k = 0,\ldots,T-1$}
        \For{$i = 1,\ldots,m$}
            \State Send $x_k^i$ to agents $j\in\mathcal{N}_k^i$ and
            receive $x_k^j$
            \State Compute the aggregation $y_k^i$ via \eqref{eq:update y}
            \State Sample $u_k^i$ and set
            $\hat{y}_k^i = y_k^i + \delta_i u_k^i$

            \For{$j = 1,\ldots,s_k^i$}
                \State Play $\hat{y}_k^i$ and query
                $J_k^i(\hat{y}_k^i,\xi_k^{i,j})$
            \EndFor

            \State Form the empirical distribution function
            $\hat{P}_k^i(z;\hat{y}_k^i)$ via \eqref{eq:EDF}
            \State Form the CVaR estimate
            $\widehat{C}(\hat{y}_k^i)$ and its gradient estimate
            $\hat{g}_k^i$ as in \eqref{eq:estimated gradient}
            \State Update $x_{k+1}^i$ via
            \eqref{eq:gradient descent}
        \EndFor
    \EndFor
\end{algorithmic}
\end{algorithm}
Define $\mathcal X_\delta:=\{x\mid x/(1-\delta_{\max}/r)\in\mathcal X\}$ as the projection set, and let $\mathcal P_{\mathcal X_\delta}(x):=\arg\min_{y\in\mathcal X_\delta}\|x-y\|^2$ denote the Euclidean projection onto this set.
Using the CVaR gradient estimate in \eqref{eq:estimated gradient}, agent $i$ updates its decision by
\begin{align}
x_{k+1}^i &= \mathcal{P}_{\mathcal{X}_\delta} [y_{k}^i - \eta_k \hat{g}_k^i]. \label{eq:gradient descent}
\end{align}   
Here, $x_0^i\in\mathcal X_\delta$ is the initial decision and $\eta_k>0$ is the step size.
Note that 
$  \mathcal{X}_\delta  \oplus \delta_i  \mathbb{B} =  \big(1-\frac{\delta_{\max}}{r}\big)\mathcal{X} \oplus \delta_i   \mathbb{B} = \big(1-\frac{\delta_{\max}}{r}\big)\mathcal{X} \oplus \frac{\delta_i}{r} r   \mathbb{B}  \subseteq \big(1-\frac{\delta_i}{r}\big)\mathcal{X} \oplus \frac{\delta_i}{r} \mathcal{X} = \mathcal{X}.  $  
Hence, by induction, $x_k^i\in\mathcal X_\delta\subset\mathcal X$ for every agent and iteration. 
Convexity of $\mathcal X_\delta$ and row stochasticity of $W_k$ further imply that $y_k^i\in\mathcal X_\delta$, 
and the construction above therefore guarantees $\hat y_k^i\in\mathcal X$.

\section{Convergence analysis}\label{sec:analysis}

This section analyzes Algorithm~\ref{alg:zeroth-order}. Section~\ref{sec:iteration} derives the one-step recursion, Section~\ref{sec:ergodic} establishes finite-time and asymptotic guarantees for the weighted ergodic iterate, and Section~\ref{sec:last iterate} proves ordinary last-iterate convergence under fixed per-agent sample sizes.
 
We impose the following assumptions, which are standard in risk-averse optimization~\cite{cardoso2019risk,wang2022risk}.
\begin{assumption}\label{assumption:Lipschitz}
For every $i\in[m]$ and $\xi^i\in\Xi_i$, the function $J^i(\cdot,\xi^i)$ is $L_i$-Lipschitz continuous on $\mathcal X$. Moreover, there exists $U_i>0$ such that $|J^i(x,\xi^i)|\le U_i$ for all $(x,\xi^i)\in\mathcal X\times\Xi_i$.
\end{assumption}
\begin{assumption}\label{assumption:convex}
For every $i\in[m]$ and $\xi^i\in\Xi_i$, the function $J^i(\cdot,\xi^i)$ is convex on $\mathcal X$.
\end{assumption}

We first present two properties of the CVaR objective and its smoothed approximation; see, e.g.,~\cite{cardoso2019risk}.  
The proofs can be seen, e.g., in \cite{cardoso2019risk}. 
\begin{lemma}[Lipschitzness]\label{lemma:smoothed cvar Lipschitz} 
Under Assumption~\ref{assumption:Lipschitz}, both $C^i$ and $C_\delta^i$ are $L_i$-Lipschitz continuous on $\mathcal X$, and $|C_\delta^i(x)-C^i(x)|\le\delta_iL_i$ for every $x\in\mathcal X$ and $i\in\mathcal V$. 
\end{lemma} 
\begin{lemma}[Convexity]\label{lemma:cvar convex}
Given Assumption \ref{assumption:convex}, $C^i(x)$ and $C_\delta^i(x)$ are both convex in $x$, for $i\in\mathcal{V}$. 
\end{lemma}

\subsection{Iteration relation}\label{sec:iteration}

This subsection bounds the network disagreement and finite-sample CVaR error and then incorporates both terms into a one-step recursion.

Define the global average of the decision state as 
$$\bar{x}_k = \frac{1}{m}\sum_{i=1}^m x_{k}^i.$$ 
Because $|J^i|\le U_i$ and $\|u_k^i\|=1$, the estimator in~\eqref{eq:estimated gradient} satisfies
\begin{align}\label{eq:gradient bound}
\|\hat{g}_k^i\| &= \norm{ \frac{d}{\delta_i}  \widehat{C}_k^i(\hat{y}_k^i) u_k^i }  \le \frac{dU_i}{\delta_i}     .
\end{align}
To simplify notation, we denote $G_i = \frac{dU_i}{\delta_i}$ for all $ i \in \mathcal{V}$, and $G_{\max} := \max_{i\in\mathcal{V}} G_i$.   
The following lemma bounds the network disagreement error driven by the doubly stochastic weight matrices.
\begin{lemma}\cite[Lemma 4.1]{sundhar2010distributed}
\label{lemma:disagreement}
Suppose Assumptions~\ref{assumption:graph q}--\ref{assumption:convex} hold. Then, for every $i\in\mathcal V$ and $k\ge0$,
\begin{align}\label{eq:disagreement}
\|\bar{x}_{k+1} - x_{k+1}^i\| 
\leq &  m\theta\mu^{k+1}  x_{0,\max} + \theta \sum_{l=1}^k \eta_{l-1} \mu^{k+1-l} \sum_{j=1}^m G_j  \nonumber\\
&+ \frac{\eta_k}{m} \sum_{j=1}^m G_j + \eta_k G_i.
\end{align}  
Here, $x_{0,\max}:=\max_{i\in\mathcal V}\|x_0^i\|$, and the graph-dependent constants $\theta>0$ and $\mu\in(0,1)$ are specified in Lemma~\ref{lemma:graph}.
\end{lemma}

Using a finite empirical distribution introduces error into the CVaR gradient estimate. We next define and bound this error.
Define the CVaR gradient estimate based on the true distribution as 
\begin{equation*}
g_k^i = \frac{d}{\delta_i}C^i(\hat{y}_k^i )u_k^i, 
\end{equation*} 
The resulting estimation error is
\begin{equation}\label{eq:cvar gradient error}
e_k^i = \hat{g}_k^i- g_k^i = \frac{d}{\delta_i}\brk{{ \widehat{C}_k^i (\hat{y}_k^i) - C^i(\hat{y}_k^i )    }u_k^i} 
\end{equation}
which is induced by the empirical CVaR approximation. 
Define $\mathcal{F}_k$ as the sigma-algebra generated by the random variables $\{x_0^i,u_t^i,\xi_t^{i,j}:i\in[m],j=1,\dots,s_t^i,t=0\dots,k-1\}$, up to iteration $k$. Then $x_k^i$ and $y_k^i$ are $\mathcal{F}_k$-measurable. Assume that, conditional on $\mathcal F_k$, the current perturbation directions $\{u_k^i\}_{i=1}^m$ are independent and each follows ${\rm Unif}(\mathbb S^{d-1})$.
Define
$\mathcal H_k := \mathcal F_k \vee \sigma(u_k^1,\ldots,u_k^m)$.
For every $i\in[m]$, the current batch is independent of $\mathcal H_k$.
Then, by \eqref{eq:two point gradient estimate}, we have
$\mathbb{E} \brk{ g_k^i | \mathcal{F}_k  } = \nabla C_\delta^i(y_k^i)$.

\begin{lemma}\label{lemma:cvar gradient error}
Suppose the sampling condition~\eqref{eq:sampling requirement} holds. 
For all $k \ge 0$, the expected global CVaR gradient estimate error with respect to $\mathcal{F}_k$ is bounded as 
\begin{align}\label{eq:cvar gradient error sum}
    &  \sum_{i=1}^m \mathbb{E} \brk{  \|e_k^i\| \big| \mathcal{F}_k  }    \le   
    md\sqrt{ 2\pi }\Big\{\max_{i\in \mathcal{V}} \frac{U_i}{\alpha_i\delta_i}\Big\} e_s.   
\end{align}
\end{lemma}
\textit{Proof.} See Appendix. 
\hfill $\qed$

The following proposition combines these estimates into the one-step recursion.
\begin{proposition}
Let Assumptions~\ref{assumption:graph q}--\ref{assumption:convex}  
and the sampling condition~\eqref{eq:sampling requirement} hold. 
For any  $z \in \mathcal{X}_{\delta}$ and all $k\ge 0$, we have
\begin{align}\label{eq:iterate relation filtration}
 &\sum_{i=1}^m  \mathbb{E}\left[\|y_{k+1}^i-z\|^2 | \mathcal{F}_k \right] \le   \sum_{i=1}^m    \|y_k^i-z\|^{2} \nonumber \\ 
&\quad  + 2\eta_kL_{\max}\sum_{j=1}^{m}\|\bar{x}_k-x_k^j\|  -2m\eta_k(\mathcal{C}(\bar{x}_k)-\mathcal{C}(z)) \nonumber \\
& \quad +  2m\eta_k D_1   
+\eta_k^{2}  \sum_{i=1}^m G_i^2,  
\end{align}
with $ D_1 =  2L_{\max}\delta_{\max}+  \sqrt{ 2\pi } dD_x\Big\{\max_{i\in \mathcal{V}} \frac{U_i}{\alpha_i\delta_i}\Big\} e_s$, 
and $L_{\max} = \max_{i \in \mathcal{V}}L_i$. 
\end{proposition}
\textit{Proof.}
Using the Euclidean projection property, for all $z \in \mathcal{X}_\delta$ and all $k$, we have 
\begin{align}\label{eq:one-step bound}
& \|x_{k+1}^i-z\|^{2} \nonumber \\
&=\| \mathcal{P}_{\mathcal{X}_\delta} \left[y_k^i-\eta_k\hat{g}_k^i\right]-z\|^{2} \nonumber \\
&\le \| y_k^i-\eta_k\hat{g}_k^i -z\|^{2} \nonumber \\
&= \|y_k^i-z\|^{2}-2 \eta_k \langle \hat{g}_k^i,y_k^i-z\rangle +\eta_k^{2}\|\hat{g}_k^i\|^{2},  
\end{align}
 where the inequality follows from $\mathcal{P}_{\mathcal{X}_\delta}[z] = z$.
Taking the conditional expectation of~\eqref{eq:one-step bound} with respect to $\mathcal F_k$ gives
\begin{align}\label{eq:exp one-step bound}
 \mathbb{E}\left[\|x_{k+1}^i-z\|^{2} |\mathcal{F}_k\right] \le &  \mathbb{E}\!\left[ \|y_k^i-z\|^{2}\big| \mathcal{F}_k\right]  \nonumber \\
&\hspace{-4em}  - 2 \eta_k\mathbb{E}\left[ \langle \hat{g}_k^i,y_k^i-z\rangle|\mathcal{F}_k\right]\! +\!\eta_k^{2}  G_i^2   , 
\end{align} 
where we used $\|\hat g_k^i\|\le G_i$ from~\eqref{eq:gradient bound}.
Moreover, since  $x_k^i$ and $y_k^i$ are $\mathcal{F}_k$-measurable, we have $\mathbb{E}\!\left[ \|y_k^i-z\|^{2}\big| \mathcal{F}_k\right]  = \|y_k^i-z\|^{2}$. 
We next relate $\|y_{k+1}^i-z\|$ to $\|x_{k+1}^i-z\|$. For $z \in \mathcal{X}_\delta$, we have  
\begin{align}\label{eq:y<=x}
&\sum_{i=1}^{m}\|y_{k+1}^i-z\|^{2}=\sum_{i=1}^{m}\left\|\sum_{j=1}^{m} w_{k+1}^{ij}  x_{k+1}^j-z\right\|^{2} \nonumber \\
&\leq \sum_{i=1}^{m} \sum_{j=1}^{m} w_{k+1}^{ij}\|x_{k+1}^j-z\|^{2} = \sum_{j=1}^{m}\|x_{k+1}^j-z\|^2.
\end{align}
The inequality follows from convexity of the squared norm, and the final equality uses column stochasticity, $\sum_{i=1}^m w_{k+1}^{ij}=1$, from Assumption~\ref{assumption:graph}.
Combining \eqref{eq:exp one-step bound} with \eqref{eq:y<=x}, 
we obtain 
\begin{align}\label{eq:iterate relation}
&\sum_{i=1}^m   \mathbb{E}\left[\|y_{k+1}^i-z\|^2|\mathcal{F}_k\right]   \le \sum_{j=1}^m  \mathbb{E}\left[\|x_{k+1}^j -z\|^2 |\mathcal{F}_k\right] \nonumber \\ 
&\!\le\! \sum_{i=1}^m  \Big\{   \|y_k^i-z\|^{2}    -2 \eta_k\mathbb{E}\left[ \langle \hat{g}_k^i,y_k^i-z\rangle|\mathcal{F}_k\right]  + \eta_k^{2}G_i^2\Big\}.
\end{align}
Using~\eqref{eq:cvar gradient error}, the inner-product term in~\eqref{eq:iterate relation} can be written as
\begin{align}\label{eq:Exp <g, y-z>}
&\mathbb{E}\left[ \langle \hat{g}_k^i,y_k^i-z\rangle|\mathcal{F}_k\right] \nonumber \\
& =\mathbb{E}\left[ \langle g_k^i, y_k^i-z\rangle |\mathcal{F}_k \right] +\mathbb{E}\left[ \langle e_k^i, y_k^i-z \rangle |\mathcal{F}_k \right] \nonumber \\ 
 &=  
 \langle   \nabla  C_\delta^i (y_k^i)      ,y_k^i-z\rangle   +\mathbb{E}\left[ \langle      e_k^i,y_k^i-z\rangle \big| \mathcal{F}_k \right]
  \nonumber \\ 
&\ge     C_\delta^i (y_k^i) -C_\delta^i(z)    -D_x \mathbb{E}\left[ \| e_k^i \| \big| \mathcal{F}_k\right] .
\end{align}
The first equality uses the error decomposition~\eqref{eq:cvar gradient error}. The inequality follows from convexity of $C_\delta^i$ (Lemma~\ref{lemma:cvar convex}) and the $\|y_k^i-z\|\le D_x$.
For $C_\delta^i (y_k^i) -C_\delta^i(z) $, we have 
\begin{align}\label{eq:C(y)-C(z)}
 &C_\delta^i(y_k^i) -C_\delta^i(z)  \nonumber \\
& =  C_\delta^i(y_k^i) - C_\delta^i(\bar{x}_k) + C_\delta^i(\bar{x}_k)-  C_\delta^i(z) \nonumber \\ 
& \ge -  L_{\max}\|\bar{x}_k-y_k^i\| + C_\delta^i(\bar{x}_k)-  C_\delta^i(z) \nonumber \\ 
& \ge - L_{\max}\|\bar{x}_k-y_k^i\| +  C^i(\bar{x}_k)-  C^i(z)  - 2L_i\delta_i.
\end{align}
The first inequality uses the $L_i$-Lipschitz continuity of $C_\delta^i$, with $L_{\max}:=\max_{i\in\mathcal V}L_i$. The second uses the smoothing bound $|C_\delta^i(x)-C^i(x)|\le\delta_iL_i$ from Lemma~\ref{lemma:smoothed cvar Lipschitz}. Because the updates remain in $\mathcal X_\delta$, their average $\bar x_k$ also belongs to $\mathcal X_\delta$.
Summing \eqref{eq:C(y)-C(z)} over $i \in [m]$, we obtain 
\begin{align}\label{eq:sum(C(y)-C(z))}
  &\sum_{i=1}^{m}  C_\delta^i(y_k^i)-C_\delta^i(z)  \nonumber \\ 
 & \! \geq \! -L_{\max}\sum_{i=1}^{m}\|\bar{x}_k-y_k^i\|+\sum_{i=1}^m\left(C^i(\bar{x}_k) -C^i(z)    -   2L_i\delta_i \right)  \nonumber \\
  & \!   \geq \! -L_{\max}\sum_{j=1}^{m}\|\bar{x}_k-x_k^j\| + \sum_{i=1}^m\left(C^i(\bar{x}_k) -C^i(z)    -   2L_i\delta_i \right)  \nonumber \\
 &  \! \geq\! -L_{\max}\sum_{j=1}^{m}\!\|\bar{x}_k-x_k^j\|\!+\!m(\mathcal{C}(\bar{x}_k)-\mathcal{C}(z))\!-\!    2mL_{\max}\delta_{\max}  
\end{align}
where the second inequality follows from
\begin{align*}
&\sum_{i=1}^{m}\|y_k^i- \bar{x}\| =\sum_{i=1}^{m}\left\|\sum_{j=1}^{m} w_k^{ij}  x_k^j- \bar{x}\right\|  \nonumber \\
&\leq \sum_{i=1}^{m} \sum_{j=1}^{m} w_k^{ij}\|x_k^j- \bar{x}\|  = \sum_{j=1}^{m}\|x_k^j- \bar{x}\| .
\end{align*} 
The last inequality follows from the definition of $\mathcal{C}(\cdot)$. 
The bound of the global estimation error $\sum_{i=1}^m\mathbb{E}\left[ \| e_k^i \| \big| \mathcal{F}_k\right] $ is provided in \eqref{eq:cvar gradient error sum}.
Finally, substituting \eqref{eq:cvar gradient error sum}, \eqref{eq:Exp <g, y-z>}, and \eqref{eq:sum(C(y)-C(z))} into \eqref{eq:iterate relation}, we obtain \eqref{eq:iterate relation filtration}.
\hfill $\qed$

\subsection{Ergodic convergence}\label{sec:ergodic}

This subsection first derives a non-asymptotic expected optimality-gap bound for the weighted ergodic iterate in 
Theorem~\ref{thm:ergodic}. 
Lemma~\ref{lemma:exact consensus} then establishes exact asymptotic consensus, 
and Corollary~\ref{cor:convergence_rate} specializes the finite-time bound to polynomially decaying step sizes 
and gives its asymptotic error order.
\begin{theorem} 
\label{thm:ergodic}
Suppose Assumptions~\ref{assumption:graph q}--\ref{assumption:convex} and the sampling condition~\eqref{eq:sampling requirement} hold. 
Then
\begin{align}\label{eq:sum c regret}
&\mathbb{E}\brk{\mathcal{C}(\hat{x}_T)- \mathcal{C}(x^\ast)}\le    \frac{D_x^2 }{2\sum_{k=0}^{T-1}\eta_{k}}+  G_{\max}^2\frac{\sum_{k=0}^{T-1}\eta_k^2}{2\sum_{k=0}^{T-1}\eta_{k}} \nonumber \\ 
&+ 
\frac{ L_{\max}}{m}\frac{\sum_{k=0}^{T-1}  \eta_k  \mathbb{E}\brk{\sum_{j=1}^{m}\|\bar{x}_k-x_k^j\|}}{\sum_{k=0}^{T-1}\eta_{k}} + 
 D_2  ,
\end{align}
where $\hat{x}_T = \frac{\sum_{k=0}^{T-1}\eta_k\bar{x}_k}{\sum_{k=0}^{T-1}\eta_k}$ is the $T$-step weighted average state, $D_2 = D_1 +   D_x\delta_{\max}L_{\max}/r $.  
\end{theorem}
\textit{Proof}.  \ 
Starting from~\eqref{eq:iterate relation filtration}, define $V_k(z):=\sum_{i=1}^m\|y_k^i-z\|^2$ for brevity.
By the tower property of conditional expectation, we have  $\mathbb{E}\left[ \mathbb{E}[ V_{k+1}(z)| \mathcal{F}_k] \right] =  \mathbb{E}\brk{ V_{k+1}(z)}$. 
Taking total expectations in~\eqref{eq:iterate relation filtration} therefore yields
\begin{align}\label{eq:iteration relation 0}
& \mathbb{E}\brk{V_{k+1}(z)  } \le   \mathbb{E}\brk{V_{k}(z) }+ 2\eta_kL_{\max}\mathbb{E}\brk{\sum_{j=1}^{m}\|\bar{x}_k-x_k^j\|} \nonumber \\ 
&-2m\eta_k\mathbb{E}\brk{\mathcal{C}(\bar{x}_k)-\mathcal{C}(z)}+  2m\eta_k D_1   
+\eta_k^{2}  \sum_{i=1}^m G_i^2.  
\end{align}
Let $x^\ast\in\arg\min_{x\in\mathcal X}\mathcal C(x)$ and define the scaled comparison point
$z_\delta:=\left(1-\delta_{\max}/r\right)x^\ast$.
Since $\mathcal X_\delta=\left(1-\delta_{\max}/r\right)\mathcal X$, we have $z_\delta\in\mathcal X_\delta$.
Taking $z=z_\delta$ in \eqref{eq:iteration relation 0}, and rearranging the terms, we obtain 
\begin{align}\label{eq:iteration relation optimal}
&2m\eta_k\mathbb{E}\brk{\mathcal{C}(\bar{x}_k)-\mathcal{C}(z_\delta)  }  \le   \mathbb{E}\brk{V_k(z_\delta)}   - \mathbb{E}\brk{V_{k+1}(z_\delta)}  \nonumber \\ 
&+2\eta_kL_{\max}   \mathbb{E}\brk{\sum_{j=1}^{m}\|\bar{x}_k-x_k^j\| }+  2m\eta_k D_1  
+\eta_k^{2}  \sum_{i=1}^m G_i^2  .
\end{align}
Summing the difference terms in~\eqref{eq:iteration relation optimal} over $k=0,\ldots,T-1$ gives
\begin{align*}
&\sum_{k=0}^{T-1} \crk{\mathbb{E}\brk{V_{k}(z_\delta)}   - \mathbb{E}\brk{V_{k+1}(z_\delta)}}  =\mathbb{E}\brk{V_{0}(z_\delta)}   - \mathbb{E}\brk{V_{T}(z_\delta)}  \nonumber\\
&\le \mathbb{E}\brk{V_{0}(z_\delta)} =  \mathbb{E}\brk{\sum_{i=1}^m  \|y_0^i-z_\delta\|^2}   \le  m D_x^2. 
\end{align*}
Then, summing~\eqref{eq:iteration relation optimal} over $k=0,\ldots,T-1$ and dividing by $2m\sum_{k=0}^{T-1}\eta_k$ gives
\begin{align}\label{eq:sum eta c regret}
& \!  \frac{\sum_{k=0}^{T-1} \eta_k\mathbb{E}\brk{\mathcal{C}(\bar{x}_k)-\mathcal{C}(z_\delta) }}{\sum_{k=0}^{T-1}\eta_k} \! \le   \!  \frac{D_x^2 }{2\sum_{k=0}^{T-1}\eta_{k}} \!+   \!\frac{G_{\max}^2}{2} \! \frac{\sum_{k=0}^{T-1}\eta_k^2}{\sum_{k=0}^{T-1}\eta_{k}} \nonumber \\ 
&+ 
\frac{ L_{\max}}{m}\frac{\sum_{k=0}^{T-1}  \eta_k  \mathbb{E}\brk{\sum_{j=1}^{m}\|\bar{x}_k-x_k^j\|}}{\sum_{k=0}^{T-1}\eta_{k}} + 
 D_1    .
\end{align}  
Since $0\in\mathcal X$, the definition of $D_x$ gives $\|x^\ast\|\leq D_x$. The Lipschitz continuity of $\mathcal C$ therefore yields
\begin{align}\label{eq:scaled-comparator-gap}
&\mathcal C(z_\delta)-\mathcal C(x^\ast)
\leq L_{\max}\|z_\delta-x^\ast\| \nonumber\\
&=\frac{L_{\max}\delta_{\max}}{r}\|x^\ast\|
\leq\frac{D_xL_{\max}\delta_{\max}}{r}.
\end{align}
Consequently,
\begin{align}\label{eq:scaled-comparator-decomposition}
&\mathcal{C}(\bar{x}_k) - \mathcal{C}(x^\ast)
=\mathcal{C}(\bar{x}_k)-\mathcal C(z_\delta)
+\mathcal C(z_\delta)-\mathcal C(x^\ast)\nonumber\\
&\leq\mathcal{C}(\bar{x}_k)-\mathcal C(z_\delta)
+\frac{D_x\delta_{\max}L_{\max}}{r}.
\end{align}

By Jensen's inequality, we have  
\begin{equation}\label{eq:jensen inequality}
\mathbb{E}\brk{\mathcal{C}(\hat{x}_T) - \mathcal{C}(x^\ast) }\le \frac{\sum_{k=0}^{T-1}\eta_{k}\mathbb{E}\brk{\mathcal{C}(\bar{x}_k)-\mathcal{C}(x^\ast)}}{\sum_{k=0}^{T-1}\eta_{k}} .  
\end{equation} 
Substituting \eqref{eq:scaled-comparator-decomposition} and \eqref{eq:jensen inequality} into \eqref{eq:sum eta c regret}, we obtain \eqref{eq:sum c regret}. 
\hfill $\blacksquare$
Theorem~\ref{thm:ergodic} provides a non-asymptotic error bound for the optimality gap of the ergodic iterate, and explicitly quantifies the effects of network disagreement, finite-sample errors in CVaR estimation, and the bias introduced by the zeroth-order approach.

Theorem~\ref{thm:ergodic} contains an accumulated disagreement term that depends on the graph and step size sequence. The next lemma characterizes this term and establishes exact asymptotic consensus.
\begin{lemma}\label{lemma:exact consensus}
Let Assumptions~\ref{assumption:graph q}--\ref{assumption:convex} hold. 
If $\sum_{k=0}^\infty \eta_k = \infty$ and $\sum_{k=0}^\infty \eta_k^2<\infty $, then $\sum_{k=0}^\infty\eta_k  \|\bar{x}_k-x_k^i\|  < \infty$, for all $i \in \mathcal{V}$. Furthermore, for all $i\in\mathcal{V}$, we have $\lim_{k\rightarrow\infty}\|\bar{x}_{k+1}-x_{k+1}^i\|=0$ almost surely.
\end{lemma}
\textit{Proof}.    See Appendix. 
\hfill $\qed$

Lemma~\ref{lemma:exact consensus} shows that, under the standard diminishing step size, the weighted disagreement is summable and the local decisions asymptotically reach exact consensus almost surely.

We next specialize the ergodic bound to polynomially decaying step sizes.   
\begin{corollary} 
\label{cor:convergence_rate}
Let Assumptions~\ref{assumption:graph q}--\ref{assumption:convex} and the sampling condition~\eqref{eq:sampling requirement}  hold. 
Choose $\eta_k=c_1/(k+1)^\beta$, where $\beta\in[0.5,1]$ and $c_1>0$. Then
\begin{equation}\label{eq:ergodic order}
\mathop{\lim\sup}_{T \to \infty} \mathbb{E}\brk{\mathcal{C}(\hat{x}_T) - \mathcal{C}^\ast} =  \order{\delta_{\max} + \frac{e_s}{\delta_{\min}}}.
\end{equation} 
\end{corollary}
\textit{Proof}.  We evaluate the terms of \eqref{eq:sum c regret} separately.
First, for the cumulated step sizes, we have 
\begin{align*}
\sum_{k=0}^{T-1} \eta_k = \left\{\begin{array}{ll}
  \order{T^{1-\beta}}    &  {\rm if}~ \beta \in [0.5, 1)\\ 
\order{\log(T)}      &  {\rm if}~ \beta = 1
\end{array}\right.,
\end{align*}
and
\begin{align*}
\sum_{k=0}^{T-1} \eta_k^2 = \left\{\begin{array}{ll}
 \order{\log(T)}    &  {\rm if}~ \beta = 0.5\\ 
\order{1}      &  {\rm if}~ \beta = (0.5, 1]
\end{array}\right. .
\end{align*}
From Lemma~\ref{lemma:disagreement},  the upper bound of the consensus error term can be written as $$\sum_{k=0}^{T-1}  \eta_k  \mathbb{E}\brk{\sum_{j=1}^{m}\|\bar{x}_k-x_k^j\|} \le K_0 + K_1G_{\max}\sum_{k=0}^{T-1}\eta_k^2,$$
where $K_0$ and $K_1$ are constants independent of $T$. 
Moreover, $G_{\max}  = \max_{i\in \mathcal{V}}\frac{d_iU_i}{\delta_i}=\order{\frac{1}{\delta_{\min}}}$, with $\delta_{\min} = \min_{i\in\mathcal{V}}\delta_i$. 
From the definition of $D_2$, we have $D_2 = \order{\delta_{\max}+ \frac{e_s}{\delta_{\min}}}$.
Substituting these estimates into~\eqref{eq:sum c regret} shows that the weighted ergodic iterate satisfies
\begin{align}\label{eq:finite-horizon order ergodic}
  \mathbb{E}\brk{\mathcal{C}(\hat{x}_T)-\mathcal{C}^\ast}  =   \order{ \frac{\rho_\beta(T)}{\delta_{\min}^2 } + \delta_{\max}+ \frac{e_s}{\delta_{\min}}} ,  
\end{align}
with 
\begin{align*}
\rho_\beta(T)
 = \left\{\begin{array}{ll}
  \frac{\log T}{\sqrt{T}}   &  {\rm if}~\beta=0.5 \\
  \frac{1}{T^{1-\beta}}&  {\rm if}~ \beta \in (0.5, 1)\\ 
\frac{1}{\log(T) }     &  {\rm if}~ \beta = 1
\end{array}\right.,
\end{align*}
Since $\lim_{T\to\infty}\rho_\beta(T)=0$, for $\beta\in[0.5,1]$, we have \eqref{eq:ergodic order}.
\hfill $\qed$

\begin{remark}
This work presents the first framework for distributed risk-averse optimization utilizing CVaR.  
Although the smoothing parameters are fixed throughout the iterations, we keep their dependence explicit in the convergence bound because they are algorithmic tuning parameters. 
In the finite-iteration result \eqref{eq:finite-horizon order ergodic}, $\rho_\beta(T)$ describes the decay of the transient optimization error.
The factor $1/\delta_{\min}^2$ arises from the magnitude of the one-point zeroth-order gradient estimator.  
Additionally, the communication network affects the finite-time bound through the accumulated disagreement term. 
However, under a diminishing step size, the disagreement vanishes asymptotically. 
This means that the network topology influences the transient convergence behavior, 
but does not change the order of the limiting error neighborhood.
\end{remark}

\begin{remark}
In the present analysis, the smoothing radii $\delta_i$ are fixed throughout the iterations. Consequently, \eqref{eq:ergodic order} shows that the expected optimality gap of the ergodic iterate approaches a nonvanishing neighborhood of the optimum, rather than establishing convergence of the algorithm itself to that neighborhood.
The term $\delta_{\max}$ represents the approximation introduced by CVaR smoothing, whereas $e_s/\delta_{\min}$ represents the finite-sample error in the empirical CVaR estimates.
Specifically, decreasing the smoothing radius reduces the approximation bias but increases the magnitude of the zeroth-order gradient estimator.
If a diminishing $\delta_k$ scheme is selected as in \cite{yi2022zeroth}, the sample size $s_k^i$ must increase asymptotically to control the resulting estimation error. We therefore use fixed smoothing radii and finite sample sizes to obtain an implementable algorithm.
\end{remark}

\subsection{Last-iterate convergence}\label{sec:last iterate}

Under the sampling condition~\eqref{eq:sampling requirement}, the sample sizes may vary over time. In this case, Theorem~\ref{thm:ergodic} and Lemma~\ref{lemma:exact consensus} yield the following last-iterate limit-inferior guarantee.

\begin{proposition}\label{prop:last-iterate-liminf}
Suppose Assumptions~\ref{assumption:graph q}--\ref{assumption:convex} and the sampling condition~\eqref{eq:sampling requirement} hold. Select $\eta_k=c_1/(k+1)^\beta$, where $\beta\in(0.5,1]$ and $c_1>0$. Then, for every agent $i\in[m]$,
\begin{equation}
\liminf_{k\to\infty}
\mathbb E\brk{\mathcal C(x_k^i)-\mathcal C^\ast}
 =\order{\delta_{\max}+\frac{e_s}{\delta_{\min}}}.
\label{eq:local-liminf-bound}
\end{equation}
\end{proposition}

\textit{Proof}. 
Combining the weighted estimate~\eqref{eq:sum eta c regret}, the comparator bound~\eqref{eq:scaled-comparator-decomposition}, and the disagreement estimate used in the proof of Corollary~\ref{cor:convergence_rate} gives
\begin{equation}
\limsup_{T\to\infty}
\frac{\sum_{k=0}^{T-1}\eta_k
\mathbb E\brk{\mathcal C(\bar x_k)-\mathcal C^\ast}}
{\sum_{k=0}^{T-1}\eta_k}
\leq D_2.
\label{eq:weighted-gap-limsup}
\end{equation}
We first show that this weighted bound implies
\begin{equation}
\liminf_{k\to\infty}
\mathbb E\brk{\mathcal C(\bar x_k)-\mathcal C^\ast}
\leq D_2.
\label{eq:average-liminf-bound}
\end{equation}
Suppose otherwise. Then there exist $\epsilon>0$ and $K\geq0$ such that
$\mathbb E[\mathcal C(\bar x_k)-\mathcal C^\ast]\geq D_2+\epsilon$ for every $k\geq K$. Because $\sum_k\eta_k=\infty$, the contribution of the finite prefix $k<K$ vanishes after division by $\sum_{k=0}^{T-1}\eta_k$. Consequently, the limit inferior of the weighted average in~\eqref{eq:weighted-gap-limsup} is at least $D_2+\epsilon$, contradicting~\eqref{eq:weighted-gap-limsup}. This proves~\eqref{eq:average-liminf-bound}.

By Lemma~\ref{lemma:exact consensus}, $\|x_k^i-\bar x_k\|\to0$ almost surely. Since both iterates lie in the compact set $\mathcal X_\delta$, dominated convergence further gives $\mathbb E\|x_k^i-\bar x_k\|\to0$. The Lipschitz continuity of $\mathcal C$ therefore yields
\begin{equation*}
\left|\mathbb E\brk{\mathcal C(x_k^i)-\mathcal C^\ast}
-\mathbb E\brk{\mathcal C(\bar x_k)-\mathcal C^\ast}\right|
\leq L_{\max}\mathbb E\|x_k^i-\bar x_k\|\to0.
\end{equation*}
Thus, the two expected gap sequences have the same limit inferior, and~\eqref{eq:local-liminf-bound} follows from~\eqref{eq:average-liminf-bound}.
\hfill $\qed$ 

Proposition~\ref{prop:last-iterate-liminf} ensures that the expected local optimality gap enters the $D_2$-neighborhood infinitely often, but it does not establish an ordinary last-iterate limit. 
We therefore impose the additional condition that each agent use a time-invariant sample size. 
The corresponding expected empirical objective is then fixed over time, which enables the following convergence result.

\begin{theorem}\label{theorem:last-iterate}
Suppose Assumptions~\ref{assumption:graph q}--\ref{assumption:convex} and the sampling condition~\eqref{eq:sampling requirement} hold. Let $s_k^i=s^i$ for every $i\in[m]$ and $k\geq0$, where the constants $s^i$ need not be identical across agents, and select $\eta_k=c_1/(k+1)^\beta$, where $\beta\in(0.5,1]$ and $c_1>0$. Then there exists a random point $\bar x_\infty\in\mathcal X_\delta$ such that $\bar x_k\to\bar x_\infty$ and $x_k^i\to\bar x_\infty$ almost surely for every $i\in[m]$. Moreover,
\begin{align}
 \lim_{k\to\infty}
\mathbb E\brk{\mathcal C(x_k^i)-\mathcal C^\ast}
= \order{\delta_{\max}+\frac{e_s}{\delta_{\min}}}.
\label{eq:last-iterate-limit}
\end{align}
\end{theorem}

\textit{Proof}.
Under $s_k^i=s^i$, define the expected smoothed empirical CVaR
\begin{align*}
\widetilde C^i(x)
 :=\mathbb E_{\nu,\xi_k^{i,1:s^i}}
\brk{\widehat C_k^i(x+\delta_i\nu)},
\quad \nu\sim{\rm Unif}(\mathbb B).
\end{align*}
The expectation is over both the smoothing perturbation and the fresh empirical batch. Because the batch size and its distribution are independent of $k$, $\widetilde C^i$ is a fixed deterministic function. The empirical CVaR is convex and $L_i$-Lipschitz for every sample realization, and these properties are preserved by smoothing and expectation. Moreover, the smoothing identity gives
\begin{equation}\label{eq:surrogate-unbiased-gradient}
\mathbb E\brk{\hat g_k^i\mid\mathcal F_k}
=\nabla \widetilde C^i(y_k^i).
\end{equation}
Define $\widetilde{\mathcal C}(x):=m^{-1}\sum_{i=1}^m\widetilde C^i(x)$,
$\widetilde{\mathcal X}^{\ast}:=\arg\min_{x\in\mathcal X_\delta}\widetilde{\mathcal C}(x)$,
$\widetilde{\mathcal C}^\ast:=\min_{x\in\mathcal X_\delta}\widetilde{\mathcal C}(x)$, and
$h_k:=\widetilde{\mathcal C}(\bar x_k)-\widetilde{\mathcal C}^\ast\geq0$.
For any $z\in\widetilde{\mathcal X}^{\ast}$, 
repeating the derivation of \eqref{eq:iterate relation filtration} with \eqref{eq:surrogate-unbiased-gradient} yields
\begin{align}
&\mathbb E\brk{V_{k+1}(z)\mid\mathcal F_k}
\leq V_k(z)
-2m\eta_k\brk{\widetilde{\mathcal C}(\bar x_k)
-\widetilde{\mathcal C}(z)}\nonumber\\
&\quad+2\eta_kL_{\max}
\sum_{j=1}^m\|\bar x_k-x_k^j\|
+\eta_k^2\sum_{j=1}^mG_j^2.
\label{eq:surrogate-recursion}
\end{align}
Lemma~\ref{lemma:exact consensus} and $\sum_{k=0}^\infty\eta_k^2<\infty$ imply that the last two terms of \eqref{eq:surrogate-recursion} are summable almost surely. Since $h_k\geq0$, the Robbins--Siegmund theorem, with descent term $2m\eta_kh_k$, shows that $V_k(z)$ converges almost surely and
\begin{equation}\label{eq:surrogate-gap-summable}
\sum_{k=0}^\infty\eta_kh_k<\infty
\quad\text{almost surely}.
\end{equation}
Since $\sum_k\eta_k=\infty$, \eqref{eq:surrogate-gap-summable} implies $\liminf_{k\to\infty}h_k=0$ almost surely. Thus, there is a subsequence along which $h_k\to0$; compactness of $\mathcal X_\delta$ and continuity of $\widetilde{\mathcal C}$ give a further subsequence converging to a point $\bar x_\infty\in\widetilde{\mathcal X}^{\ast}$.

We next pass from the convergent subsequence to the full sequence. Let $\mathcal D$ be a countable dense subset of the nonempty compact set $\widetilde{\mathcal X}^{\ast}$. For each fixed $z\in\mathcal D$, the Robbins--Siegmund argument above shows that $V_k(z)$ converges almost surely. Countability of $\mathcal D$ therefore gives an event on which this convergence holds for every $z\in\mathcal D$ with probability 1. 
 Double stochasticity gives
$m^{-1}\sum_i y_k^i=\bar x_k$ and hence
\begin{equation*}
V_k(z)=m\|\bar x_k-z\|^2
+\sum_{i=1}^m\|y_k^i-\bar x_k\|^2.
\end{equation*}
Moreover, Lemma~\ref{lemma:exact consensus} implies 
$\sum_{i=1}^m\|y_k^i-\bar x_k\|
\leq\sum_{j=1}^m\|x_k^j-\bar x_k\|$, 
which converges to zero almost surely.   
Thus, on this common event, $\|\bar x_k-z\|$ converges for every $z\in\mathcal D$. Because $z\mapsto\|\bar x_k-z\|$ is uniformly $1$-Lipschitz, density extends this convergence to every $z\in\widetilde{\mathcal X}^{\ast}$, including the random cluster point $\bar x_\infty$. Along the convergent subsequence, $\|\bar x_k-\bar x_\infty\|\to0$; hence the full sequence satisfies $\bar x_k\to\bar x_\infty\in\widetilde{\mathcal X}^{\ast}$ almost surely.
Lemma~\ref{lemma:exact consensus} consequently gives
$x_k^i\to\bar x_\infty$ almost surely for every $i\in[m]$.

It remains to bound the true-objective gap at this surrogate optimizer. Define $\mathcal C_\delta(x):=m^{-1}\sum_{i=1}^mC_\delta^i(x)$. The same calculation as in Lemma~\ref{lemma:cvar gradient error}, followed by the sampling condition~\eqref{eq:sampling requirement}, gives, for every $x\in\mathcal X_\delta$,
\begin{equation}\label{eq:surrogate-gradient-error}
\left\|\nabla\widetilde{\mathcal C}(x)
-\nabla\mathcal C_\delta(x)\right\|
\leq   d\sqrt{2\pi }
\max_{i\in\mathcal V}\frac{U_i}{\alpha_i\delta_i}\,e_s.
\end{equation}
Let $z_\delta=(1-\delta_{\max}/r)x^\ast$. The first-order optimality condition for $\bar x_\infty\in\widetilde{\mathcal X}^{\ast}$ gives
$\langle\nabla\widetilde{\mathcal C}(\bar x_\infty),
\bar x_\infty-z_\delta\rangle\leq0$. Thus, by convexity of $\mathcal C_\delta$ and \eqref{eq:surrogate-gradient-error},
\begin{align*}
&\mathcal C_\delta(\bar x_\infty)-\mathcal C_\delta(z_\delta)
\leq\left\langle\nabla\mathcal C_\delta(\bar x_\infty),
\bar x_\infty-z_\delta\right\rangle\\
&=\left\langle\nabla\mathcal C_\delta(\bar x_\infty)
-\nabla\widetilde{\mathcal C}(\bar x_\infty),
\bar x_\infty-z_\delta\right\rangle\\
&\quad+\left\langle\nabla\widetilde{\mathcal C}(\bar x_\infty),
\bar x_\infty-z_\delta\right\rangle\\
&\leq D_x\left\|\nabla\mathcal C_\delta(\bar x_\infty)
-\nabla\widetilde{\mathcal C}(\bar x_\infty)\right\|\\
&\leq   D_xd\sqrt{ 2\pi } 
\max_{i\in\mathcal V}\frac{U_i}{\alpha_i\delta_i}\,e_s.
\end{align*}
Finally, $|\mathcal C_\delta(x)-\mathcal C(x)|\leq L_{\max}\delta_{\max}$ and \eqref{eq:scaled-comparator-gap} imply
\begin{align}
\mathcal C(\bar x_\infty)-\mathcal C^\ast
&\leq\left|\mathcal C(\bar x_\infty)-\mathcal C_\delta(\bar x_\infty)\right|
+\mathcal C_\delta(\bar x_\infty)-\mathcal C_\delta(z_\delta)\nonumber\\
&\quad+\left|\mathcal C_\delta(z_\delta)-\mathcal C(z_\delta)\right|
+\mathcal C(z_\delta)-\mathcal C^\ast\nonumber\\
&\leq2L_{\max}\delta_{\max}
+  D_xd\sqrt{2\pi } 
\max_{i\in\mathcal V}\frac{U_i}{\alpha_i\delta_i}\,e_s\nonumber\\
&\quad+\frac{D_xL_{\max}\delta_{\max}}{r}
=D_2.
\label{eq:last-iterate-true-gap}
\end{align}
By its definition, $D_2=\order{\delta_{\max}+e_s/\delta_{\min}}$. Since $\mathcal C$ is continuous and bounded on the compact set $\mathcal X$, dominated convergence gives
\begin{equation*}
\lim_{k\to\infty}\mathbb E\brk{\mathcal C(x_k^i)-\mathcal C^\ast}
=\mathbb E\brk{\mathcal C(\bar x_\infty)-\mathcal C^\ast}
\leq D_2,
\end{equation*}
which proves \eqref{eq:last-iterate-limit}.
\hfill 
$\qed$

\begin{remark}
In the risk-neutral setting, distributed stochastic approximation can achieve almost-sure convergence to a common optimizer under diminishing step sizes and suitably controlled subgradient errors~\cite{sundhar2010distributed}. More recent work studies decentralized stochastic gradient descent with changing topologies and local updates~\cite{koloskova2020unified}, stochastic gradient tracking~\cite{pu2021distributed}, and asynchronous optimization over graphs~\cite{even2024asynchronous}.
These methods typically use first-order stochastic gradients of expected-loss objectives, and their bounds characterize the effects of gradient noise, data heterogeneity, and network communication. 
In contrast, Theorem~\ref{theorem:last-iterate} uses finite-sample CVaR estimates and fixed zeroth-order smoothing. 
 Consequently, the local iterates converge to an optimizer of the expected 
 empirical-smoothed surrogate $\widetilde{\mathcal C}$, 
 while their limiting true CVaR gap is bounded by $D_2=\order{\delta_{\max}+e_s/\delta_{\min}}$. 
 This residual neighborhood is caused by CVaR estimation and smoothing.
\end{remark}

\section{Centralized Benchmark}\label{sec:centralize}

This section introduces a centralized benchmark for comparison with the distributed method. It targets the same objective~\eqref{eq problem}, but a coordinator aggregates all local zeroth-order CVaR gradient estimates before applying a single update. We first describe the algorithm and then establish its ergodic and fixed-sample last-iterate guarantees.

\subsection{Centralized Algorithm}\label{sec:centralize alg}
At iteration $k$, the central coordinator maintains a single decision variable $x_k^{\mathrm{cen}}$. 
Each agent $i\in\mathcal V$ samples a direction $u_k^i\sim{\rm Unif}(\mathbb S^{d-1})$, 
forms $\hat{x}_{i,k}^{\rm cen}=x_k^{\rm cen}+\delta_i u_k^i$, 
and queries $J^i(\hat{x}_{i,k}^{\rm cen},\xi_k^{i,j})$ for $j=1,\ldots,s_k^i$. 
The sample sizes satisfy~\eqref{eq:sampling requirement}. 
The resulting evaluations define the empirical distribution $\hat P_k^i(z;\hat{x}_{i,k}^{\rm cen})$ as in~\eqref{eq:EDF}, 
the empirical CVaR value $\widehat C_k^i(\hat{x}_{i,k}^{\rm cen})$, and the local zeroth-order CVaR gradient estimate: 
\begin{equation}\label{eq:cent gradient estimate i}
\hat{g}_{i,k}^{\mathrm{cen}} = \frac{d}{\delta_i}
\widehat{C}_k^i (\hat{x}_{i,k}^{\rm cen})
u_k^i.
\end{equation}
The central coordinator aggregates these estimates as
\begin{equation}\label{eq:cent gradient estimate}
\hat{g}_k^{\mathrm{cen}} = 
\frac{1}{m}\sum_{i=1}^m \hat{g}_{i,k}^{\mathrm{cen}}, 
\end{equation}
and updates the decision variable via
\begin{equation}\label{eq:cent gradient descent}
x_{k+1}^{\mathrm{cen}} = 
\mathcal{P}_{\mathcal{X}_\delta}
\left[
x_k^{\mathrm{cen}}-\eta_k \hat{g}_k^{\mathrm{cen}}
\right],  
\end{equation}
with initial value $x_0^{\rm cen}\in\mathcal X_\delta$.
The zeroth-order centralized risk-averse learning algorithm is summarized in Algorithm~\ref{alg:centralize}. 
\begin{algorithm}[t]
\caption{Zeroth-order centralized risk-averse learning}
\label{alg:centralize}
\begin{algorithmic}[1]
    \Require Initial value $x_0^{\mathrm{cen}}$, iteration horizon $T$,
    smoothing parameters $\delta_i$, learning rates $\eta_k$, and risk
    levels $\alpha_i$

    \For{$k = 0,\ldots,T-1$}
        \State The coordinator broadcasts $x_k^{\mathrm{cen}}$ to all agents

        \For{$i = 1,\ldots,m$}
            \State Sample $u_k^i$ and set
            $\hat{x}_{i,k}^{\mathrm{cen}}
            = x_k^{\mathrm{cen}}+\delta_i u_k^i$

            \For{$j = 1,\ldots,s_k^i$}
                \State Play $\hat{x}_{i,k}^{\mathrm{cen}}$ and query
                $J_k^i(\hat{x}_{i,k}^{\mathrm{cen}},\xi_k^{i,j})$
            \EndFor

            \State Form the empirical distribution function
            $\hat{P}_k^i(z;\hat{x}_{i,k}^{\mathrm{cen}})$
            via \eqref{eq:EDF}
            \State Form the CVaR estimate
            $\widehat{C}(\hat{x}_{i,k}^{\mathrm{cen}})$ and its gradient
            estimate $\hat{g}_{i,k}^{\mathrm{cen}}$ as in
            \eqref{eq:cent gradient estimate i}
        \EndFor

        \State The coordinator aggregates the gradient estimates
        $\hat{g}_k^{\mathrm{cen}}$ as in
        \eqref{eq:cent gradient estimate}
        \State Update $x_{k+1}^{\mathrm{cen}}$ via
        \eqref{eq:cent gradient descent}
    \EndFor
\end{algorithmic}
\end{algorithm}

\subsection{Convergence analysis}\label{sec:centralize convergence}
We next analyze the centralized benchmark.  Because the coordinator aggregates all local estimates before updating, no network-disagreement term appears, while finite-sample CVaR error and smoothing bias remain.
Corollary~\ref{coro:centralize ergodic} gives finite-time and asymptotic-order guarantees for the weighted ergodic iterate, while Corollary~\ref{coro:centralize last-iter} establishes last-iterate convergence under fixed per-agent sample sizes.
\begin{corollary}\label{coro:centralize ergodic}
Let Assumptions~\ref{assumption:Lipschitz}--\ref{assumption:convex}, and the sampling condition~\ref{eq:sampling requirement} hold. Then, the centralized method satisfies
\begin{align}\label{eq:centralize deviation result}
&\mathbb{E}[\mathcal{C}(\hat{x}_{T}^{\text{cen}}) - \mathcal{C}(x^\ast)] \! \le \! \frac{D_{x}^{2}}{2 \sum_{k=0}^{T-1} \eta_{k}}  \!+ \! \frac{G_{\text{cen}}^{2} \sum_{k=0}^{T-1} \eta_{k}^{2}}{2 \sum_{k=0}^{T-1} \eta_{k}} \! + \! D_2  ,
\end{align}
with $G_{\text{cen}} =\frac{1}{m} \sum_{i=1}^{m} G_i$, and with $\hat{x}_{T}^{\text{cen}} = \frac{\sum_{k=0}^{T-1} \eta_{k} x_k^{\rm cen}}{\sum_{k=0}^{T-1} \eta_{k}}$.
By choosing $\eta_{k} = \frac{c_1}{(k+1)^\beta} $, with $\ \beta \in [0.5, 1]$ and  $c_1 >0$, we have 
\begin{equation*}
\mathbb{E}[\mathcal{C}(\hat{x}_T^{\text{cen}}) - \mathcal{C}(x^\ast)] = \order{\frac{\rho_\beta(T)}{\delta_{\min}^2} + \delta_{\max} + \frac{e_s}{\delta_{\min}}} .
\end{equation*}
Letting $T\to\infty$ yields
\begin{equation*}
 \mathop{\lim\sup}_{T\to \infty}\mathbb{E}[\mathcal{C}(\hat{x}_T^{\text{cen}}) - \mathcal{C}(x^\ast)] = \order{  \delta_{\max} + \frac{e_s}{\delta_{\min}}}   .
\end{equation*}
\end{corollary}
\textit{Proof.} 
See Appendix.
\hfill $\qed$
\begin{corollary}\label{coro:centralize last-iter}
Let Assumptions~\ref{assumption:Lipschitz}--\ref{assumption:convex} 
and the sampling condition~\eqref{eq:sampling requirement} hold, 
let $s_k^i=s^i$ for every $i$ and $k$, and select $\eta_k=c_1/(k+1)^\beta$, 
where $\beta\in(0.5,1]$ and $c_1>0$. 
Then there exists a random point $x_\infty^{\rm cen}\in \mathcal X_\delta$ 
such that $x_k^{\rm cen}\to x_\infty^{\rm cen}$ almost surely and
\begin{align*}
 \lim_{k\to\infty}\mathbb{E}[\mathcal C(x_k^{\rm cen})-\mathcal C^\ast] 
=\order{\delta_{\max}+e_s/\delta_{\min}}. 
\end{align*}
\end{corollary}
\textit{Proof.} 
See Appendix.
\hfill $\qed$
\begin{remark}   
The distributed and centralized finite-time bounds have the same dependence on $T$, the smoothing radii, and the sampling accuracy. However, they are not equal because the distributed bound contains graph-dependent constants arising from network disagreement, whereas the centralized bound does not.
\end{remark}

\section{Simulations}\label{sec:simulation}
This section evaluates the proposed method on a distributed estimation problem~\cite{xu2017convergence} and examines the effects of graph connectivity, smoothing, and sample size.
Each sensor observes an unknown parameter $x^{\mathrm{true}}\in\mathcal X:=\{x\in\mathbb R^{10}:\|x\|_\infty\leq10\}$. At iteration $k$, its $j$th independent measurement is
\begin{equation*}
z_k^{i,j}=A_i x^{\mathrm{true}}+w_k^{i,j},
\end{equation*}
where $A_i\in\mathbb R^{10\times10}$ is fixed throughout the experiment. The entries of $A_i$ and the components of $x^{\mathrm{true}}$ are initially sampled from $\mathcal N(0,1)$, with $x^{\mathrm{true}}$ projected onto $\mathcal X$. The components of $w_k^{i,j}$ are sampled independently from $\mathcal N(0,0.01^2)$ truncated to $[-10,10]$. Consequently, the measurement distribution has bounded support and, since $\mathcal X$ is compact and the matrices $A_i$ are fixed, the simulated loss is bounded.
For a candidate decision $x$, agent $i$ uses the regularized least-squares loss
\begin{equation*}
J^i(x,z^i)=\frac12\|z^i-A_i x\|^2+\frac{\lambda}{2}\|x\|^2,
\end{equation*}
where $\lambda=10^{-4}$ is the regularization parameter. We therefore solve
\begin{align*}
\mathop{\min}_{x\in\mathcal X}\ \frac1m\sum_{i=1}^m
{\rm CVaR}_{\alpha_i}\brk{J^i(x,z^i)}.
\end{align*}
We set $\alpha_i=0.5$ for every agent. 
We compute the reference solution $x^\ast$ by minimizing an empirical CVaR objective constructed from 8192 measurements per agent.
At every iteration, 
the CVaR empirical gap $\widehat{\mathcal C}_k(x)-\widehat{\mathcal C}_k(x^\ast)$ 
is evaluated on the sampled batch, 
with $x=\bar x_k$ for the distributed method and $x=x_k^{\rm cen}$ for the centralized method. 
It is therefore an empirical performance measure rather than the true CVaR optimality gap
and may occasionally be negative.
Each experiment consists of 20 independent Monte Carlo trials. 
The solid lines represent the trial averages, 
and the shaded regions depict $\pm1$ standard deviation.

Figure~\ref{fig:graph} compares Algorithm~\ref{alg:zeroth-order} across several static graphs and against the centralized benchmark in Algorithm~\ref{alg:centralize}. The network contains 16 agents and uses one of four undirected topologies: (1) an Erd\H{o}s--R\'enyi graph with edge probability $p=0.4$;
(2) a ring graph, in which each agent is connected to exactly two neighbors; 
(3) a two-dimensional grid graph; and (4) a complete graph.
The matrix $W_k$ is
constructed using the Metropolis rule and is doubly stochastic. Specifically,
its entries are given by
\begin{align*}
w_k^{ij} =\left\{
\begin{array}{ll}
\dfrac{1}{\max\{d_k^i,d_k^j\}+1},
& (i,j)\in\mathcal{E}_k, i\neq j \\
  1-\displaystyle\sum_{j\in\mathcal{N}_k^i\setminus\{i\}}w_k^{ij},
& i=j,\\
0,
& \text{otherwise},
\end{array}\right., 
 \end{align*}
where $d_k^i$ denotes the degree of agent $i$ at time $k$. 
We use $\delta^i=0.4$, $\eta_k=0.004/(k+1)^{0.55}$, $s_k^i=64$, and 10000 iterations. 
We compare the methods using the consensus error $m^{-1}\sum_{i=1}^{m}\|x_k^i-\bar{x}_k\|^2$, 
optimization error $\|\bar{x}_k-x^\ast\|^2$, total-state error $m^{-1}\sum_{i=1}^{m}\|x_k^i-x^\ast\|^2$, 
and the paired empirical CVaR gap defined above. 
For the centralized method, the optimization and total-state errors both reduce to $\|x_k^{\rm cen}-x^\ast\|^2$.  
The results indicate that stronger connectivity primarily improves the transient consensus rate: the complete and Erd\H{o}s--R\'enyi graphs reach consensus faster than the grid and ring graphs, with the ring exhibiting the slowest decay.
For the complete graph, the Metropolis matrix equals $\frac{1}{m}\mathbf 1\mathbf 1^\top$, 
so every agent evaluates its local estimate from the common network average after each mixing step. 
Its optimization, total-state, and empirical CVaR-gap curves therefore nearly overlap the centralized curves 
and are omitted from panels~(b)--(d) for visual clarity. 
The two methods need not coincide exactly because the distributed method projects the local updates separately, 
whereas the centralized method projects their aggregated update. 

Fig.~\ref{fig:connectivity} evaluates Algorithm~\ref{alg:zeroth-order} 
over periodic time-varying communication networks with window lengths $q\in\{1,2,5,10\}$. 
For each $q$, the edges of a complete graph are partitioned among $q$ consecutive instantaneous graphs. 
Hence, the union of every $q$ consecutive graphs is complete. 
We use $\delta^i=1$, $\eta_k=0.01/(k+1)^{0.55}$, $s_k^i=256$, 
and 5000 iterations. 
The trajectories remain stable even when individual communication graphs are temporarily disconnected. As shown in Fig.~\ref{fig:connectivity}, increasing $q$ generally slows the transient decay of both the total state error and the empirical CVaR gap,  with the clearest degradation occurring at $q=10$; the $q=1$ and $q=2$ curves remain close and occasionally cross.

\begin{figure}[t]
\centering
\subfigure[Consensus error]{\label{fig:graph_consensus}
\includegraphics[width=0.234\textwidth]{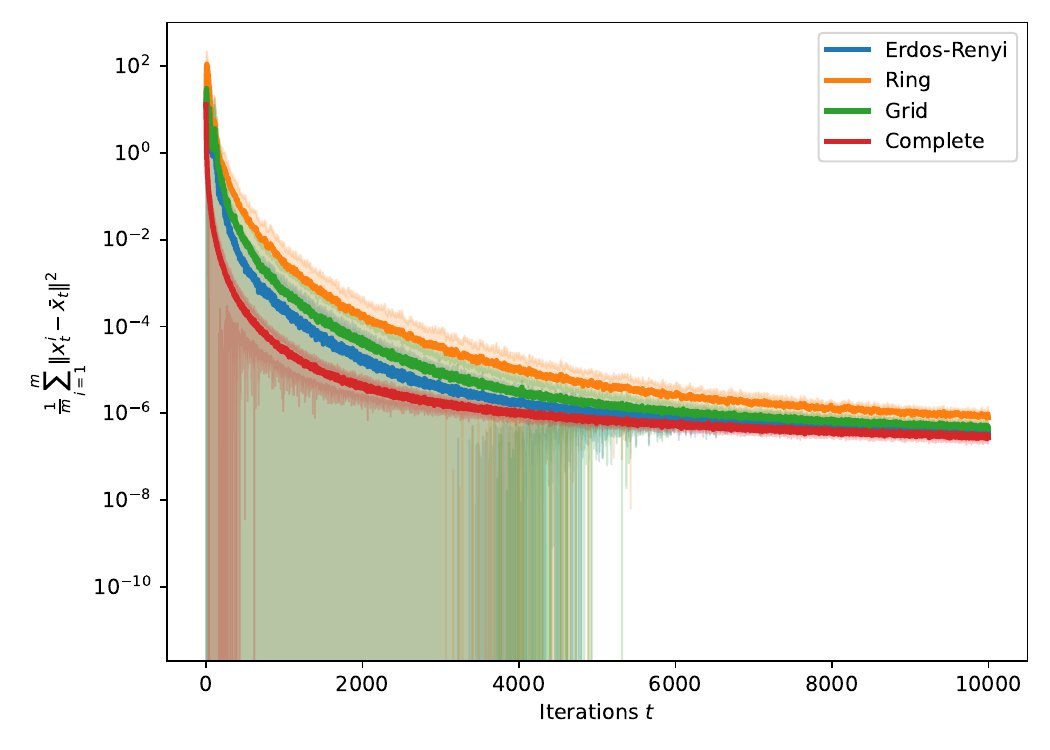}}
\hspace{-7pt}
\subfigure[Optimization error]{\label{fig:graph_optimization}
\includegraphics[width=0.234\textwidth]{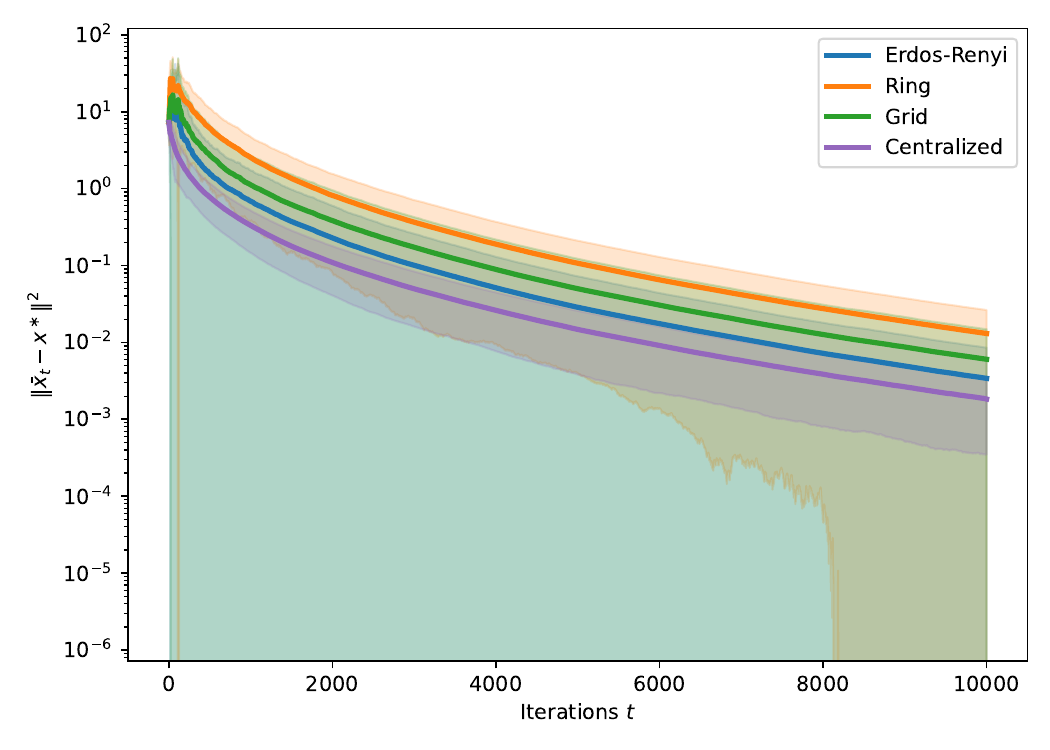}}
\vspace{-7pt}
\subfigure[Total state error]{\label{fig:graph_state}
\includegraphics[width=0.23\textwidth]{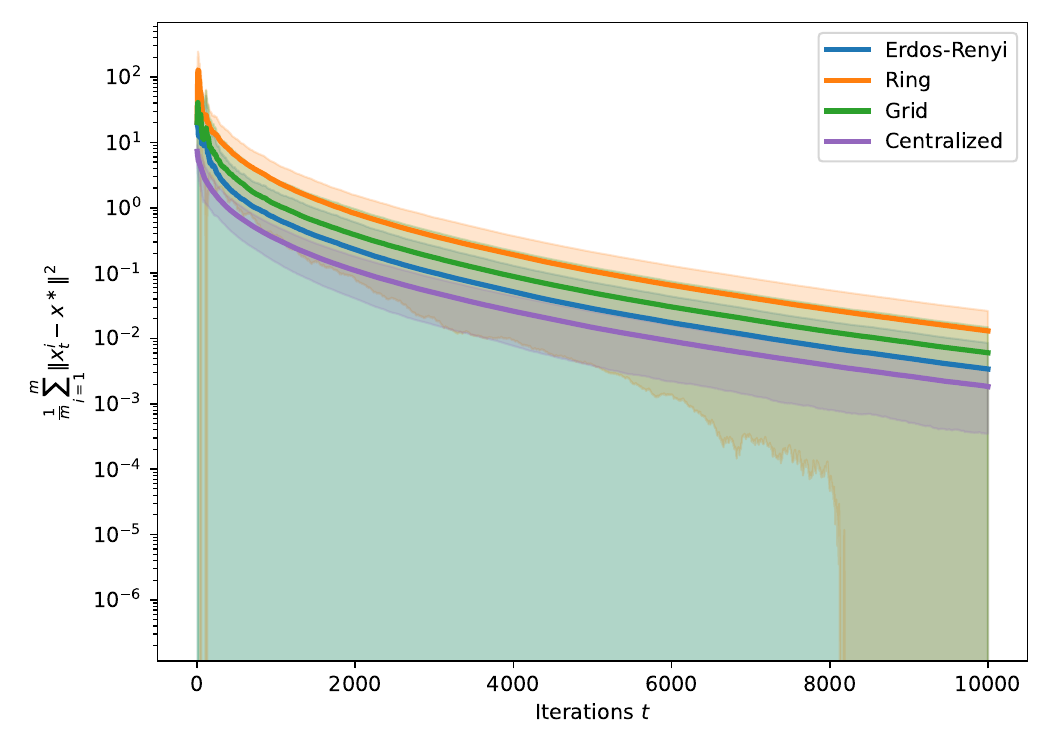}} 
\hspace{-5pt}
\subfigure[ Empirical CVaR gap ]{\label{fig:graph_cvar}
\includegraphics[width=0.23\textwidth]{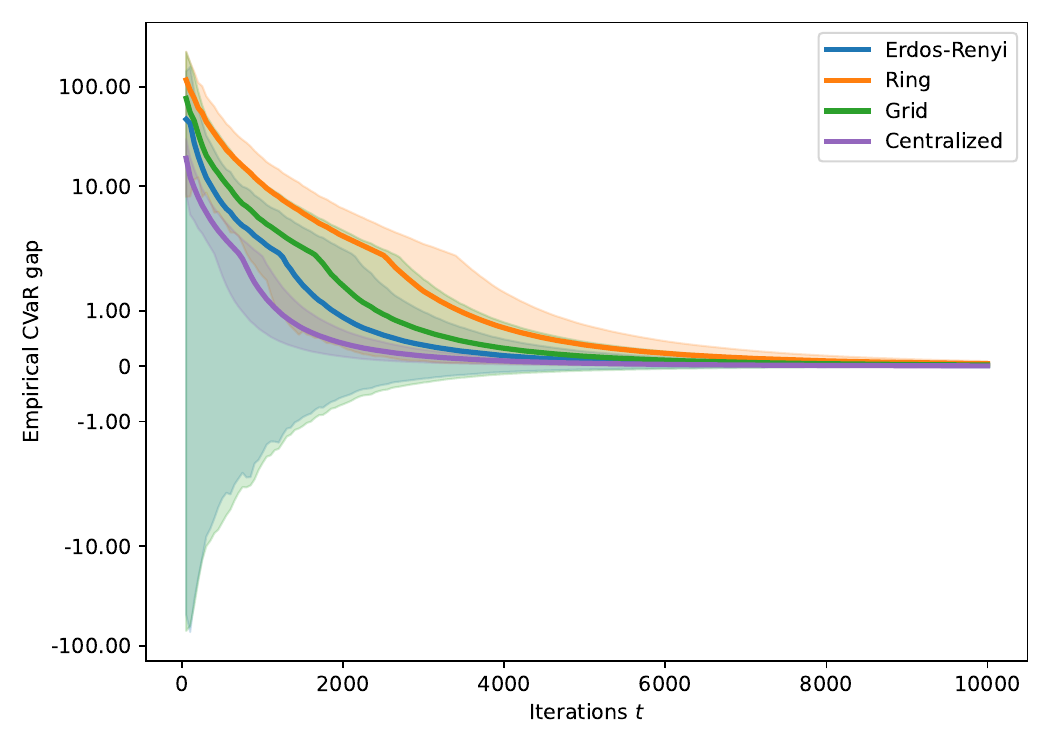}}
\vspace{-5pt}
\caption{ Performance under different communication graphs and comparison with the centralized benchmark. 
Panel~(a) includes all four distributed topologies. 
Panels~(b)--(d) omit the complete-graph curves because they nearly overlap the centralized curves. 
The centralized benchmark represents full information aggregation.
The Erd\H{o}s--R'enyi, grid, and ring graphs exhibit progressively weaker connectivity.} \label{fig:graph}
\end{figure}

\begin{figure}[t]
\centering
\subfigure[Total state error]{\label{fig:connect_state}
\includegraphics[width=0.23\textwidth]{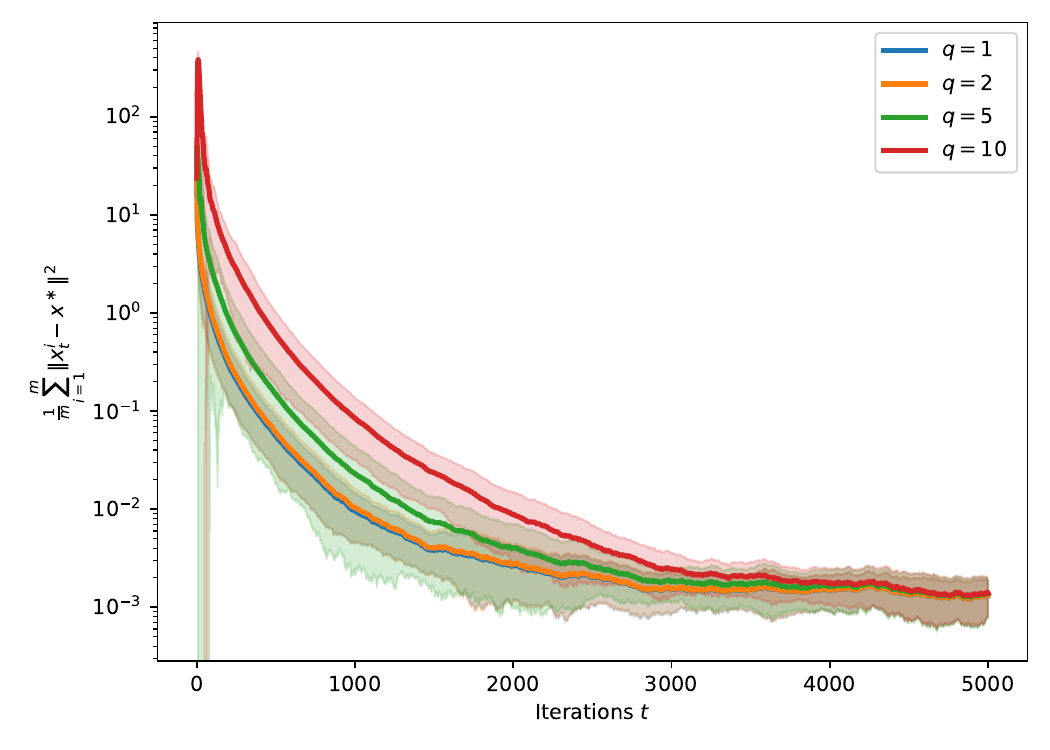}}
\hspace{-5pt}
\subfigure[Empirical CVaR gap]{\label{fig:connect_cvar}
\includegraphics[width=0.23\textwidth]{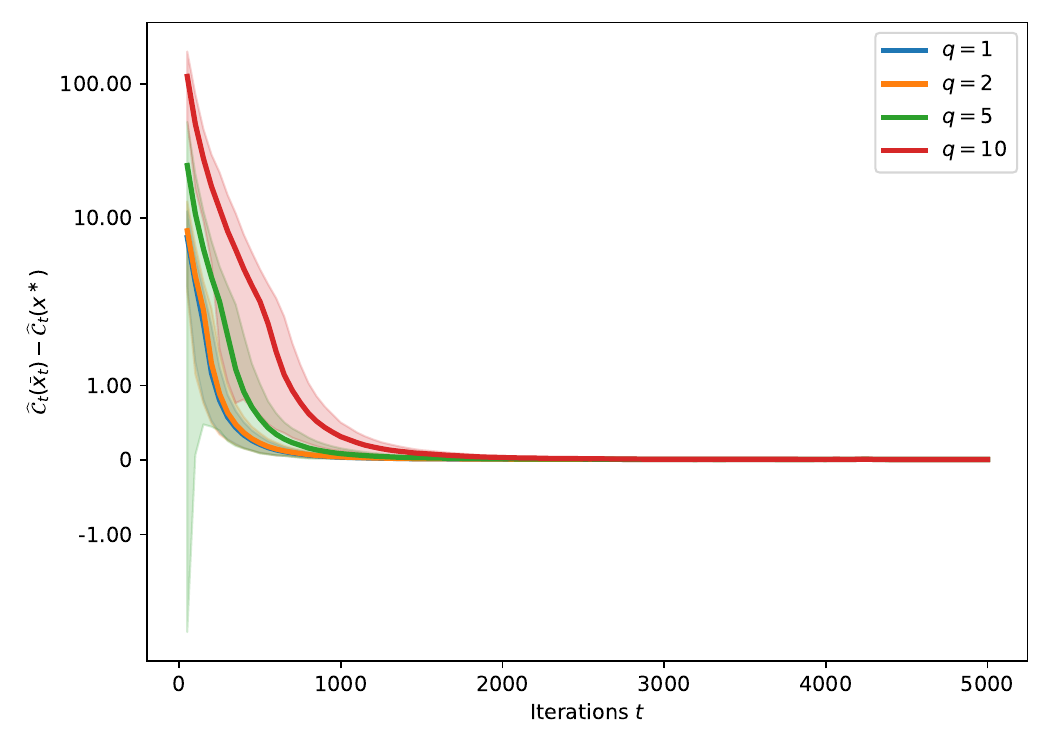}}
\vspace{-5pt}
\caption{ Performance of Algorithm~\ref{alg:zeroth-order} 
over periodic graph sequences whose $q$-step union is complete. 
Larger connectivity windows generally slow transient convergence.}\label{fig:connectivity}
\end{figure}

\begin{figure}[t]
\centering
\subfigure[Total state error]{\label{fig:delta_state}
\includegraphics[width=0.23\textwidth]{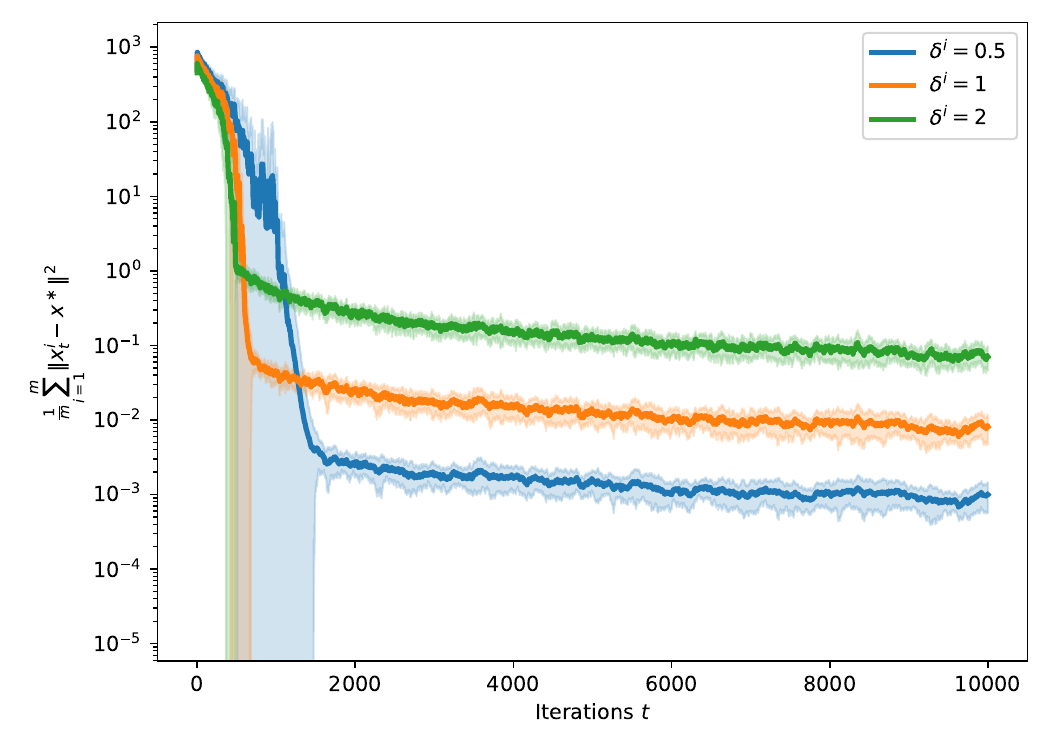}}
\hspace{-7pt}
\subfigure[ Empirical CVaR gap ]{\label{fig:delta_cvar}
\includegraphics[width=0.23\textwidth]{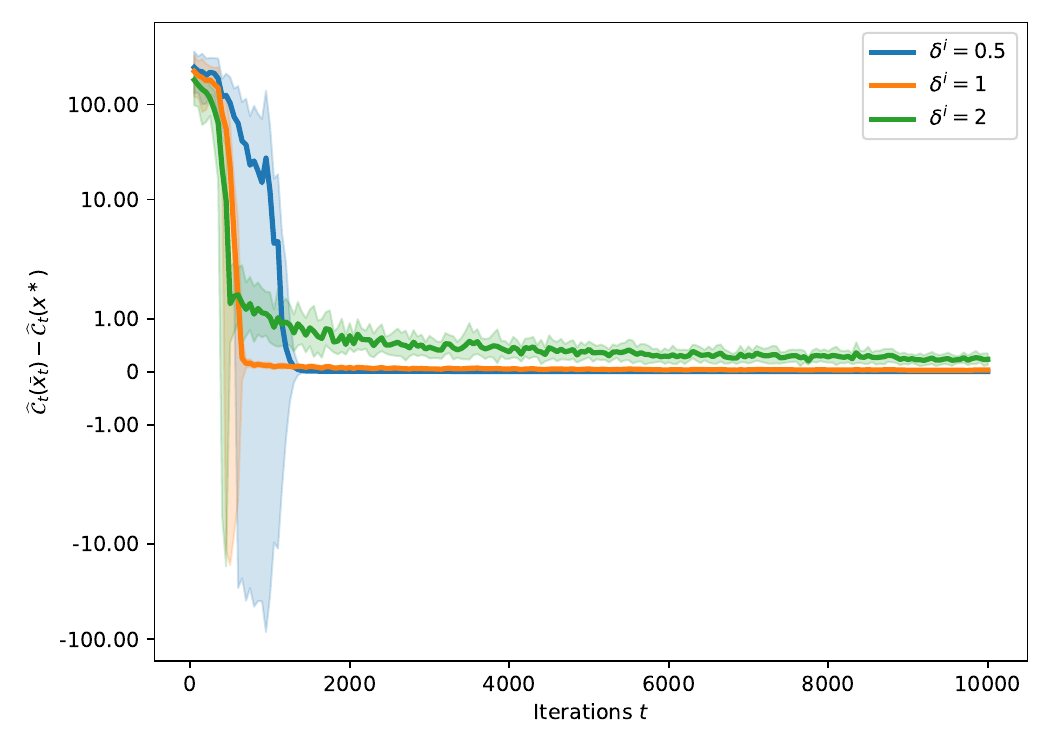}}
\vspace{-10pt}
\caption{Effect of the smoothing radius over a fixed Erd\H{o}s--Rényi graph.
Larger radii with correspondingly larger step-size coefficients accelerate transient convergence, 
but lead to larger long-run error neighborhoods.}
\label{fig:delta}
\vspace{-10pt}
\end{figure}

\begin{figure}[t]
\centering
\subfigure[Total state error]{\label{fig:sample_state}
\includegraphics[width=0.23\textwidth]{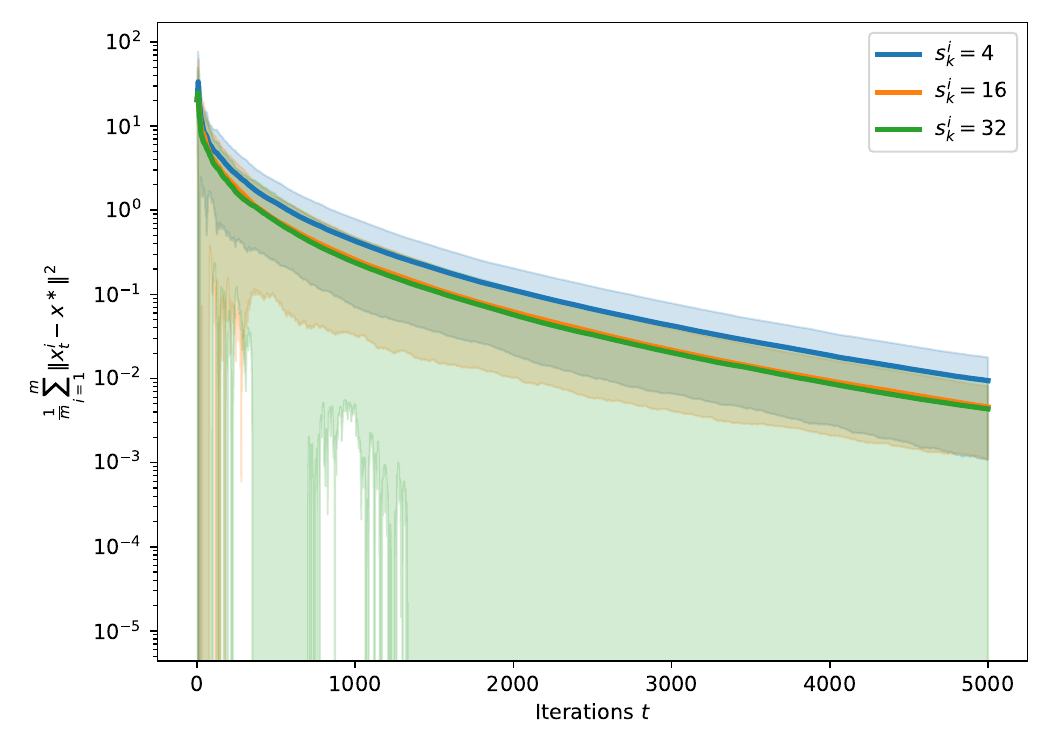}}
\hspace{-7pt}
\subfigure[ Empirical CVaR gap]{\label{fig:sample_cvar}
\includegraphics[width=0.23\textwidth]{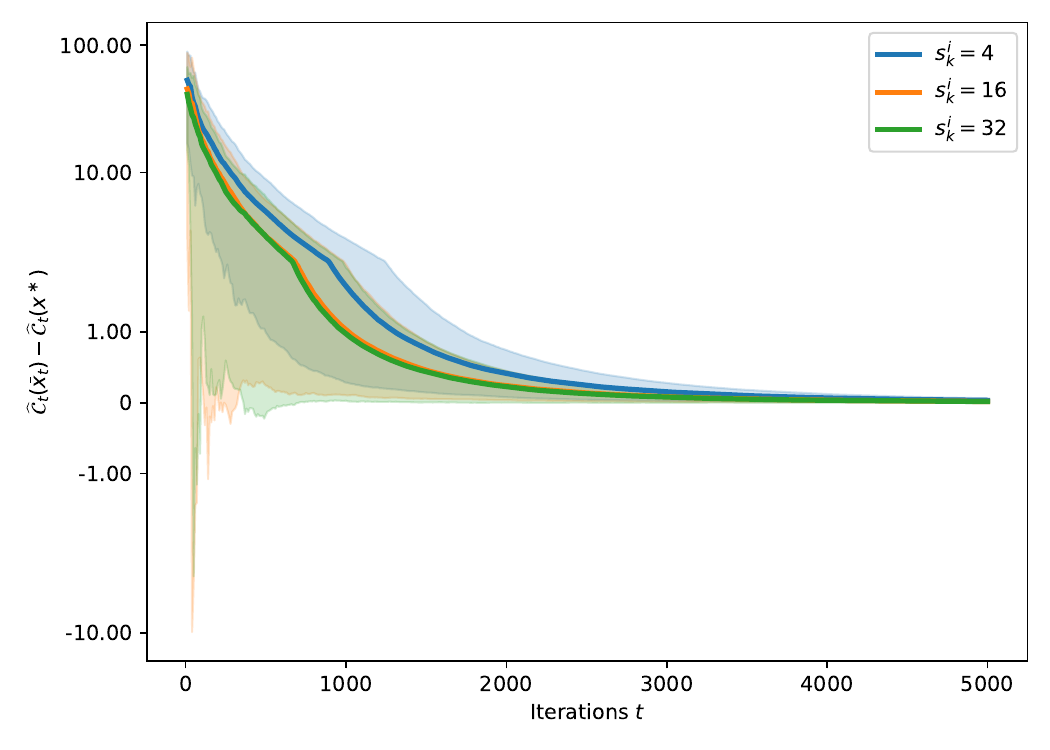}}
\vspace{-7pt}
\caption{ Effect of the CVaR sample size over a fixed Erd\H{o}s--Rényi graph. 
A larger sample size reduces the finite-sample variability of the CVaR-based zeroth-order estimator 
and accelerates transient convergence.} 
\label{fig:sample}
\vspace{-7pt}
\end{figure}

In Figs.~\ref{fig:delta} and~\ref{fig:sample}, 
we further investigate the effects of the smoothing radius and sample size. 
Both experiments use the same fixed Erdős--Rényi graph model with edge probability $p=0.4$ as in Fig.~\ref{fig:graph}. 
Fig.~\ref{fig:delta} considers $\delta^i\in\{0.5,1,2\}$ 
with the corresponding step-size coefficients $c_0\in\{0.05,0.1,0.2\}$, respectively, $\beta=0.55$, $s_k^i=256$, 
and 10000 iterations. 
The results reveal a trade-off between transient speed and asymptotic accuracy. 
Larger smoothing radii reduce the variance of the one-point estimator and, 
together with the correspondingly larger step-size coefficients, accelerate the transient decrease. 
However, they also lead to larger long-run error neighborhoods, consistent with the smoothing-bias term in the theoretical bound.

Fig.~\ref{fig:sample} compares $s_k^i\in\{4,16,32\}$ for every agent, using $\delta^i=0.5$, 
$\eta_k=0.005/(k+1)^{0.55}$, and 5000 iterations. 
Over this finite horizon, larger sample sizes reduce the variability of the empirical CVaR estimator, 
accelerate the transient decrease, and yield lower total-state error levels. 
The empirical CVaR-gap curves approach zero and become close at later iterations.
This behavior is qualitatively consistent with the dependence on sample size established in Theorems~\ref{thm:ergodic}--\ref{theorem:last-iterate}.

\section{Conclusion}\label{sec:5_conc}

This paper developed a zeroth-order method for distributed risk-averse stochastic convex optimization over time-varying networks. 
Using only noisy local loss evaluations, 
each agent constructs an empirical CVaR value and a one-point estimate of a smoothed CVaR gradient. 
Under standard convexity, Lipschitz continuity, and connectivity assumptions, 
we established a finite-time expected optimality gap bound for the weighted ergodic iterate and proved exact asymptotic consensus. 
For time-varying sample sizes, the expected local last-iterate gaps satisfy a limit-inferior guarantee. 
Furthermore, with fixed sample sizes, 
the local iterates converge almost surely to a common optimizer of the expected empirical smoothed surrogate, 
and their limiting expected true CVaR gaps are of order $\delta_{\max}+e_s/\delta_{\min}$.
The centralized benchmark has the same qualitative dependence on the horizon, smoothing radii, 
and sampling accuracy, 
whereas the distributed finite-time bound also contains graph-dependent constants arising from network disagreement. 
To the best of our knowledge, 
this is the first study of distributed stochastic convex optimization with CVaR objectives over networks. The simulations illustrate the efficacy of the method.

\section{Appendix}\label{sec:appendix}
This appendix collects auxiliary CVaR and consensus results and provides the proofs omitted from the main text. 
 \begin{remark} 
 The formulation allows each agent to have an individual risk level $\alpha_i$. If all agents instead share a common level $\alpha$, subadditivity of CVaR gives ${\rm CVaR}_\alpha[\sum_{i=1}^mJ^i]\le\sum_{i=1}^m{\rm CVaR}_\alpha[J^i]$.
The equality holds if and only if the random variables $J_1, J_2, \dots, J_m$ are positively correlated and move entirely in tandem. A trivial special case is that the random variables are exactly identical.  
In quantitative finance, the difference $ \sum_{i=1}^m \text{CVaR}_\alpha(J^i) - \text{CVaR}_\alpha\left(\sum_{i=1}^m J^i\right)$ is known as the diversification price, which represents the cost the system incurs to enable distributed computation \cite{mcneil2015quantitative}.
\end{remark}

The following lemma bounds the difference between two CVaR values in terms of the corresponding distribution functions. 
It will be used to control the finite-sample gradient-estimation error.
\begin{lemma}\cite[Lemma 3]{wang2022risk}\label{lemma:cvar-estimation error bound} 
Let $F$ and $G$ be the cumulative distribution functions of two random variables supported on $[-U,U]$. 
Given a risk level $\alpha$, we have 
\begin{equation*} 
    \Big|\text{CVaR}_\alpha[F] - \text{CVaR}_\alpha [G] \Big| \le \frac{2U}{\alpha} \sup_{x} \big|F(x)-G(x)\big|. 
\end{equation*} 
\end{lemma}
\textit{Proof of Lemma~\ref{lemma:cvar gradient error}.}
We bound the aggregate conditional error $\sum_{i=1}^m\mathbb E[\|e_k^i\|\mid\mathcal F_k]$. Because $\hat y_k^i$ depends on the current perturbation, define $\mathcal H_k:=\mathcal F_k\vee\sigma(u_k^1,\ldots,u_k^m)$.
The Dvoretzky--Kiefer--Wolfowitz inequality~\cite{dvoretzky1956asymptotic} gives, for every $t\ge0$,
\begin{equation*}
    \mathbb{P}\left\{ \sup_{z} |\hat{P}_k^i(z;\hat{y}_k^i) - P_k^i(z;\hat{y}_k^i)| \ge t | \mathcal{H}_k \right\}  \le   2e^{-2s_k^it^2}, 
\end{equation*}
for all $i$. 
Therefore, \begin{align*}
& \mathbb{E} \brk{\sup_{z} |\hat{P}_k^i(z;\hat{y}_k^i) - P_k^i(z;\hat{y}_k^i)|  \big| \mathcal{H}_k }   \nonumber\\
&\le \int_{t=0}^\infty 2  e^{-2s_k^it^2} {\rm d} t = \sqrt{\frac{\pi}{2s_k^i}}.
\end{align*}
Combining this estimate with Lemma~\ref{lemma:cvar-estimation error bound} yields
\begin{align}\label{eq: cvar error}
    \norm{ \widehat{C}_k^i(\hat{y}_k^i) - C^i(\hat{y}_k^i)} \!\le\!  \frac{2U_i}{\alpha_i}\! \crk{\!\sup_{z} |\hat{P}_k^i(z;\hat{y}_k^i) - P_k^i(z;\hat{y}_k^i)|}.
\end{align}
From the definition of $e_k^i$, we have
\begin{align*}
    & \mathbb{E}\brk{\|e_k^i \| | \mathcal{H}_k} = \mathbb{E}\brk{ \norm{ \frac{d}{\delta_i}\prt{\widehat{C}_k^i(\hat{y}_k^i) - C^i(\hat{y}_k^i)}u_k^i} \bigg| \mathcal{H}_k} \\ 
    &\le \frac{d}{\delta_i} \mathbb{E}\brk{\norm{ \widehat{C}_k^i(\hat{y}_k^i) - C^i(\hat{y}_k^i)} \bigg| \mathcal{H}_k}  \|u_k^i\| \le  \frac{  dU_i}{\alpha_i\delta_i} \sqrt{\frac{2\pi}{s_k^i}}.
\end{align*}  
The second inequality uses~\eqref{eq: cvar error} and $\|u_k^i\|=1$. Taking the conditional expectation with respect to $\mathcal F_k$ and applying the tower property gives
\begin{equation}\label{eq:DKW cvar mismatch F}
\mathbb{E}\brk{\mathbb{E}\brk{\|e_k^i\| | \mathcal{H}_k } | \mathcal{F}_k} =  \mathbb{E}\brk{\|e_k^i\| | \mathcal{F}_k}  \le  \frac{ 2dU_i}{\alpha_i\delta_i}  \sqrt{\frac{\pi}{2s_k^i}}.
\end{equation}
Summing~\eqref{eq:DKW cvar mismatch F} over $i\in[m]$ and using~\eqref{eq:sampling requirement} yields
\begin{align*}
&\sum_{i=1}^m  \mathbb{E}\brk{       \|e_k^i\| \big| \mathcal{F}_k }   
   \le  \sum_{i=1}^m\frac{  dU_i}{\alpha_i\delta_i} \sqrt{\frac{2\pi}{s_k^i}} \nonumber \\
&  \le    md\sqrt{2 \pi  }\Big\{\max_{i\in \mathcal{V}} \frac{U_i}{\alpha_i\delta_i}\Big\} e_s   .
\end{align*}
\hfill $\blacksquare$

The next lemma collects standard properties of geometrically weighted sequences used in the proof of Lemma~\ref{lemma:exact consensus}.
\begin{lemma}\cite[Lemma 3.1]{nedic2008distributed}\label{lemma:accumulative sequence}
    Let $\{\gamma_k\}_{k\ge0}$ be a scalar sequence.
\begin{enumerate}
    \item  If $\lim_{k \to \infty} \gamma_k = \gamma$ and $0 < \rho < 1$, then \\$\lim_{k \to \infty} \sum_{l=0}^k \rho^{k-l} \gamma_l = \frac{\gamma}{1-\rho}$.
    
    \item If $\gamma_k \geq 0$ for all $k$, $\sum_k \gamma_k < \infty$ and $0 < \rho < 1$, then $\sum_{k=0}^\infty \left( \sum_{l=0}^k \rho^{k-l} \gamma_l \right) < \infty$.
    
    \item  If $\limsup_{k \to \infty} \gamma_k = \gamma$ and $\{\zeta_k\}$ is a positive scalar sequence with $\sum_{k=1}^\infty \zeta_k = \infty$, then $\limsup_{K \to \infty} \frac{\sum_{k=0}^K \gamma_k \zeta_k}{\sum_{k=0}^K \zeta_k} \leq \gamma$. In addition, if $\liminf_{k \to \infty} \gamma_k = \gamma$, then $\lim_{K \to \infty} \frac{\sum_{k=0}^K \gamma_k \zeta_k}{\sum_{k=0}^K \zeta_k} = \gamma$.
\end{enumerate}
\end{lemma}

Under Assumption~\ref{assumption:graph}, every $W_k$ is doubly stochastic. For $k\ge s$, define $\Phi(k,s):=W_kW_{k-1}\cdots W_{s+1}$, let $[\Phi(k,s)]_{i,j}$ denote its $(i,j)$ entry, and let $\bm e\in\mathbb R^m$ be the column vector with all entries equal to 1. The following lemma gives the geometric mixing bound used in the consensus analysis.
\begin{lemma}\cite[Lemma 3.2]{nedic2008distributed}
\label{lemma:graph}
Let Assumptions~\ref{assumption:graph q} and~\ref{assumption:graph} hold. Then
\begin{enumerate}
    \item $\lim_{k \to \infty} \Phi(k, s) = \frac{1}{m} \bm{e}\bm{e}^\top$ for all $s $.
    \item \textit{Further, the convergence is geometric and the rate of convergence is given by}
    $
    \left| [\Phi(k, s)]_{i,j} - \frac{1}{m} \right| \leq \theta \mu^{k-s}, 
    $ where  
    $
    \theta = \left( 1 - \frac{\varrho}{4m^2} \right)^{-2}$, $\mu = \left( 1 - \frac{\varrho}{4m^2} \right)^{\frac{1}{q}}.
    $
\end{enumerate}
\end{lemma}

\noindent\textit{Proof of Lemma~\ref{lemma:exact consensus}.} 
\label{app:decision convergence}
We first relate $\|\bar{x}_{k+1}-x_{k+1}^j\|$ to $\|\bar{x}_k-x_k^j\|$, for all $j\in\mathcal{V}$ and $k$.
Define an auxiliary variable 
$r_{k+1}^i = \sum_{j=1}^m w_k^{ij}x_k^j - \eta_{k} \hat{g}_{k}^i.$
Note that $\mathcal{P}_{\mathcal{X}_\delta}[r_{k+1}^i] = x_{k+1}^i$. 
Because $\bar x_{k+1}=m^{-1}\sum_{i=1}^m x_{k+1}^i$ minimizes the sum of squared distances to the vectors $\{x_{k+1}^i\}_{i=1}^m$,
\begin{align}\label{eq:x_k+1 <= r_k+1}
&\sum_{i=1}^m \|x_{k+1}^i - \bar{x}_{k+1}\|^2 \nonumber \\
&\leq \sum_{i=1}^m \|x_{k+1}^i - \bar{x}_k\|^2 \le \sum_{i=1}^m \|r_{k+1}^i - \bar{x}_k\|^2.
\end{align}
The final inequality follows because $\bar x_k\in\mathcal X_\delta$ and Euclidean projection onto $\mathcal X_\delta$ is nonexpansive, so $\|x_{k+1}^i-\bar x_k\|\le\|r_{k+1}^i-\bar x_k\|$.

We next relate $\sum_{i=1}^m \|r_{k+1}^i - \bar{x}_k\|^2$ to $\sum_{i=1}^m \|x_k^i - \bar{x}_k\|^2$. 
From the definition of $r_{k+1}^i$, we have
\begin{align}\label{eq:r_k+1 <= x_k}
&\sum_{i=1}^m\|r_{k+1}^i - \bar{x}_k\|^2\nonumber \\ 
&\leq \sum_{i=1}^m\Big\{ \sum_{j=1}^m w_k^{ij} \|x_k^j - \bar{x}_k\|^2 + \eta_k^2 G_i^2  \nonumber \\ 
&\hspace{1em}+ 2\eta_k G_i \sum_{j=1}^m w_k^{ij} \|x_k^j - \bar{x}_k\|  \Big\} \nonumber \\ 
&\leq  \sum_{j=1}^m \|x_k^j - \bar{x}_k\|^2  + \eta_k^2 \sum_{i=1}^m G_i^2   \nonumber \\  
&\hspace{1em}+ 2\eta_k \sum_{i=1}^m G_i \sum_{j=1}^m w_k^{ij}  \|x_k^j - \bar{x}_k\|.
\end{align}
The second inequality uses column stochasticity, $\sum_{i=1}^m w_k^{ij}=1$. Substituting~\eqref{eq:r_k+1 <= x_k} into~\eqref{eq:x_k+1 <= r_k+1}, and rearraging the terms gives
\begin{align}\label{eq:exp r_k+1 <= x_k}
&\sum_{i=1}^m \|x_{k+1}^i - \bar{x}_{k+1}\|^2       \leq   \sum_{j=1}^m \|x_k^j - \bar{x}_k\|^2\nonumber\\
&\hspace{1em}+ \eta_k^2 \sum_{i=1}^m  G_i^2  + 2\eta_k \sum_{i=1}^m  G_i \sum_{j=1}^m  w_k^{ij}   \|x_k^j - \bar{x}_k\|.  
\end{align}
We next bound the last term of \eqref{eq:exp r_k+1 <= x_k}.
Multiplying~\eqref{eq:disagreement} by $\eta_{k+1}$ gives
\begin{align*}
& \eta_{k+1}\|\bar{x}_{k+1}-x_{k+1}^j\| \nonumber \\
&\le  \eta_{k+1}\bigg(m\theta\mu^{k+1} \|x_{0,\max}\| + \theta \sum_{l=1}^k \eta_{l-1} \mu^{k+1-l} \sum_{i=1}^m G_i \nonumber \\
& \hspace{1em} + \frac{\eta_k}{m} \sum_{i=1}^m G_i + \eta_k G_j  \bigg).
\end{align*}
By the inequality $\eta_{k+1} \eta_{l-1} G_i \leq \frac{1}{2} \left( \eta_{k+1}^2 + \eta_{l-1}^2 G_i^2 \right)$, the preceding inequality yields 
\begin{align}\label{eq:eta*disagreement}
& \eta_{k+1} \|\bar{x}_{k+1}-x_{k+1}^j\|  \leq  \eta_{k+1} m\theta\mu^{k+1} \|x_{0,\max}\|  \nonumber \\ 
& \quad + \frac{\theta }{2} \sum_{l=1}^k \mu^{k+1-l} \sum_{i=1}^m \left( \eta_{k+1}^2 + \eta_{l-1}^2 G_i^2 \right)\nonumber \\
& \quad+ \frac{1}{2m} \sum_{i=1}^m \left(  \eta_{k+1}^2 + \eta_{k}^2  G_i^2 \right) + \frac{1}{2}( \eta_{k+1}^2 + \eta_k^2  G_j^2) \nonumber \\ 
 & \leq \eta_{k+1} m \theta \mu^{k+1}  \|x_{0,\max}\| + \left(1 + \frac{m\theta\mu}{2(1-\mu)}\right) \eta_{k+1}^2 \nonumber \\
& \quad + \theta \sum_{l=1}^k \eta_
{l-1}^2 \mu^{k+1-l} \sum_{i=1}^m G_i^2   + \frac{1}{m} \eta_k^2 \sum_{i=1}^m G_i^2 + \eta_k^2 G_j^2.
\end{align}
The second inequality uses $\sum_{l=0}^k\mu^{k+1-l}\le\mu/(1-\mu)$ and $1/(2m)+1/2\le1$. Because $\{\eta_k\}$ is bounded and $\sum_{k=0}^\infty\eta_k^2<\infty$, all nonconvolution terms in~\eqref{eq:eta*disagreement} are summable. Lemma~\ref{lemma:accumulative sequence} further gives $\sum_{k=0}^\infty\sum_{l=1}^k\eta_{l-1}^2\mu^{k+1-l}<\infty$, so the convolution term is also summable. Therefore,
\begin{align}
    \sum_{k=0}^\infty  \eta_{k+1} \|\bar{x}_{k+1}-x_{k+1}^j\|   < \infty. \label{eq:finite}
\end{align}

Define
$
X_k=\sum_{i=1}^m\|x_k^i-\bar x_k\|^2 .
$
From \eqref{eq:r_k+1 <= x_k}, we have
$
X_{k+1}-X_k\le h_k,
$
where
$
h_k= 
\eta_k^2\sum_{i=1}^mG_i^2
+
2\eta_k\sum_{i=1}^mG_i
\sum_{j=1}^m w_k^{ij}\|x_k^j-\bar x_k\|.
$
Since the right-hand side is summable, i.e.,
$
\sum_{k=1}^{\infty}h_k<\infty,
$
By~\eqref{eq:finite} and $\sum_{k=1}^\infty\eta_k^2<\infty$, the sequence $\{h_k\}$ is summable. A quasi-Fej\'er argument~\cite{combettes2001quasi} then shows that $X_k$ converges to a finite nonnegative limit.
Moreover, 
\begin{align*}
&\sum_{i=1}^mG_i
\sum_{j=1}^m w_k^{ij}\|x_k^j-\bar x_k\| \le G_{\max}\sum_{i=1}^m 
\sum_{j=1}^m w_k^{ij}\|x_k^j-\bar x_k\| \\
&\le G_{\max}\sum_{j=1}^m \|x_k^j-\bar x_k\|.
\end{align*}
If $X_\infty>0$, then $\sum_{i=1}^m\|x_k^i-\bar x_k\|\ge\sqrt{X_k}\ge\sqrt{X_\infty/2}$ for all sufficiently large $k$. Since $\sum_k\eta_k=\infty$, this would imply $\sum_{k=K}^\infty\eta_k\sum_{i=1}^m\|x_k^i-\bar x_k\|=\infty$, contradicting~\eqref{eq:finite}. Hence $X_\infty=0$.
\hfill $\blacksquare$

\textit{Proof of Corollary~\ref{coro:centralize ergodic}.}
Define the exact-CVaR gradient estimate and its empirical estimation error by
\begin{align*}
&g_{i,k}^{\rm cen} = \frac{d}{\delta_{i}} C_{i}(x_k^{\rm cen} + \delta_{i} u_{k}^{i}) u_{k}^{i}, \nonumber \\ 
&e_{i,k}^{\rm cen} = \hat{g}_{i,k}^{\rm cen} - g_{i,k}^{\rm cen},
\end{align*}
The aggregate estimator decomposes as
$$\hat{g}_k^{\rm cen} = \frac{1}{m} \sum_{i=1}^{m} g_{i,k}^{\rm cen} + \frac{1}{m} \sum_{i=1}^{m} e_{i,k}^{\rm cen}.$$ 
Define $e_k^{\rm cen}:=m^{-1}\sum_{i=1}^m e_{i,k}^{\rm cen}$ and let $\mathcal F_k^{\rm cen}$ be the history generated by $\{x_0^{\rm cen},u_t^i,\xi_t^{i,j}:i\in[m],\ j=1,\ldots,s_t^i,\ t=0,\ldots,k-1\}$. Also define the smoothed centralized objective $\mathcal C_\delta(x):=m^{-1}\sum_{i=1}^mC_\delta^i(x)$.
Since
$\mathbb{E}[g_{i,k}^{\rm cen} \mid \mathcal{F}_{k}^{\rm cen}] = \nabla C_\delta^i(x_k^{\rm cen})$, it follows that
\begin{equation*}
\mathbb{E}[\hat{g}_k^{\rm cen} \mid \mathcal{F}_{k}^{\rm cen}] = \nabla \mathcal{C}_\delta(x_k^{\rm cen}) + \mathbb{E}[e_k^{\rm cen} \mid \mathcal{F}_{k}^{\rm cen}]   .
\end{equation*}
For any $z\in\mathcal X_\delta$, nonexpansiveness of the projection gives
\begin{align*}
&\|x_{k+1}^{\rm cen} - z\|^{2} \le \|x_k^{\rm cen} - \eta_{k} \hat{g}_k^{\rm cen} - z\|^{2} \\ 
&= \|x_k^{\rm cen} - z\|^{2} - 2\eta_{k} \langle \hat{g}_k^{\rm cen}, x_k^{\rm cen} - z \rangle + \eta_{k}^{2} \|\hat{g}_k^{\rm cen}\|^{2} .
\end{align*} 
Taking the conditional expectation with respect to $\mathcal F_k^{\rm cen}$ yields
\begin{align}\label{eq:centralize one-step}
&\mathbb{E}[\|x_{k+1}^{\rm cen} - z\|^{2} | \mathcal{F}_{k}^{\rm cen}] \le \|x_k^{\rm cen} - z\|^{2}  + \eta_{k}^{2} G_{\text{cen}}^{2} \nonumber \\ 
&- 2\eta_{k} \langle \nabla \mathcal{C}_\delta(x_k^{\rm cen}), x_k^{\rm cen} - z \rangle+ 2\eta_{k} D_{x}  \mathbb{E}\brk{\|e_k^{\rm cen}\|  |\mathcal{F}_{k}^{\rm cen}}. 
\end{align}
 Here we used 
$\|\hat{g}_{k}^{\rm cen}\| \le \frac{1}{m}\sum_{i=1}^m \|\hat{g}_{i,k}^{\rm cen}\| \le  \frac{1}{m}\sum_{i=1}^m  G_{i} = G_{\rm cen},$
and the $\mathcal F_k^{\rm cen}$-measurability of $x_k^{\rm cen}$.
Convexity of $\mathcal C_\delta$ gives
\begin{align}\label{eq:centralize CVaR deviation}
&\langle \nabla \mathcal{C}_\delta(x_k^{\rm cen}), x_k^{\rm cen} - z \rangle \ge \mathcal{C}_\delta(x_k^{\rm cen}) - \mathcal{C}_\delta(z) \nonumber \\ 
& \ge \mathcal{C}(x_k^{\rm cen}) - \mathcal{C}(z) - 2L_{\max} \delta_{\max},
\end{align}
where the second inequality follows from $|\mathcal{C}_\delta(x) - \mathcal{C}(x)| \le L_{\max} \delta_{\max}$.

Moreover, by following the proof of Lemma~\ref{lemma:cvar gradient error}, the centralized CVaR gradient error can be bounded as
\begin{align}\label{eq:centralize estimate error}
&   \mathbb{E}[\|e_k^{\rm cen}\| \mid \mathcal{F}_{k}^{\rm cen}] \le \frac{1}{m} \sum_{i=1}^{m} \mathbb{E}[\|e_{i,k}^{\rm cen}\| \mid \mathcal{F}_{k}^{\rm cen}] \nonumber \\ 
&\le   d\sqrt{2\pi } \left\{ \max_{i \in \mathcal{V}} 
\frac{U_{i}}{\alpha_{i} \delta_{i}} \right\} \frac{1}{m} \sum_{i=1}^{m} \frac{1}{\sqrt{s_{k}^{i}}} 
 \nonumber \\ 
&\le d\sqrt{2\pi } \left\{ \max_{i \in \mathcal{V}} \frac{U_{i}}{\alpha_{i} \delta_{i}} \right\} e_s .
\end{align} 
Substituting~\eqref{eq:centralize CVaR deviation} and~\eqref{eq:centralize estimate error} into~\eqref{eq:centralize one-step} and taking total expectations yields
\begin{align}\label{eq:centralize one-step 2}
&   \mathbb{E}[\|x_{k+1}^{\rm cen} - z\|^{2}] \le \mathbb{E}[\|x_k^{\rm cen} - z\|^{2}] - 2\eta_{k} \mathbb{E}[\mathcal{C}(x_k^{\rm cen}) - \mathcal{C}(z)] \nonumber \\ 
&+  2\eta_{k}D_1 + \eta_{k}^{2} G_{\text{cen}}^{2}.
\end{align}
Set $z_\delta:=(1-\delta_{\max}/r)x^\ast\in\mathcal X_\delta$. Substituting $z=z_\delta$ into~\eqref{eq:centralize one-step 2} and rearranging gives
\begin{align*}
&2\eta_{k} \mathbb{E}[\mathcal{C}(x_k^{\rm cen}) - \mathcal{C}(z_\delta)] \le \mathbb{E}\brk{\|x_k^{\rm cen} - z_\delta\|^{2}} \nonumber\\
&- \mathbb{E}\brk{\|x_{k+1}^{\rm cen} - z_\delta\|^{2}}+ 2\eta_{k}D_1   + \eta_{k}^{2} G_{\text{cen}}^{2}    .
\end{align*}
Summing this inequality over $k=0,\ldots,T-1$ yields 
\begin{align*}
& 2 \sum_{k=0}^{T-1} \eta_{k} \mathbb{E}[\mathcal{C}(x_k^{\rm cen}) - \mathcal{C}(z_\delta)] \nonumber \\
 &\le D_{x}^{2} + 2D_1  \sum_{k=0}^{T-1} \eta_{k} + G_{\text{cen}}^{2} \sum_{k=0}^{T-1} \eta_{k}^{2}.   
\end{align*}
By convexity and Jensen's inequality, we have 
$$\mathbb{E}[\mathcal{C}(\hat{x}_{T}^{\text{cen}}) - \mathcal{C}(z_\delta)] \le \frac{\sum_{k=0}^{T-1} \eta_{k} \mathbb{E}[\mathcal{C}(x_k^{\rm cen}) - \mathcal{C}(z_\delta)]}{\sum_{k=0}^{T-1} \eta_{k}}.$$
Hence,
$$\mathbb{E}[\mathcal{C}(\hat{x}_{T}^{\text{cen}}) - \mathcal{C}(z_\delta)] \le \frac{D_{x}^{2}}{2 \sum_{k=0}^{T-1} \eta_{k}} + \frac{G_{\text{cen}}^{2} \sum_{k=0}^{T-1} \eta_{k}^{2}}{2 \sum_{k=0}^{T-1} \eta_{k}} + D_1.  $$ 
Furthermore, by \eqref{eq:scaled-comparator-gap}, $\mathcal C(z_\delta)-\mathcal C(x^\ast)\le D_xL_{\max}\delta_{\max}/r$, we obtain \eqref{eq:centralize deviation result}. The remaining follows the proof of Corollary~\ref{cor:convergence_rate}.
\hfill $\blacksquare$

\textit{Proof of Corollary~\ref{coro:centralize last-iter}.}
Under the fixed-sample-size condition, \eqref{eq:surrogate-unbiased-gradient} gives
\begin{equation*}
\mathbb E[\hat g_k^{\rm cen}\mid\mathcal F_k^{\rm cen}]
=\nabla\widetilde{\mathcal C}(x_k^{\rm cen}).
\end{equation*}
For any $z\in\widetilde{\mathcal X}^{\ast}$, non-expansiveness of the projection and convexity of $\widetilde{\mathcal C}$ therefore yield
\begin{align*}
&\mathbb E[\|x_{k+1}^{\rm cen}-z\|^2\mid\mathcal F_k^{\rm cen}]
\leq\|x_k^{\rm cen}-z\|^2\\
&\quad-2\eta_k\brk{\widetilde{\mathcal C}(x_k^{\rm cen})
-\widetilde{\mathcal C}(z)}
+\eta_k^2G_{\rm cen}^2.
\end{align*}
The Robbins--Siegmund theorem, $\sum_k\eta_k=\infty$, and $\sum_k\eta_k^2<\infty$ imply that $x_k^{\rm cen}$ converges almost surely to a random point $x_\infty^{\rm cen}$. Applying the same stochastic quasi-Fejér argument as in the proof of Theorem 2 yields $x_k^{\rm cen}\to x_\infty^{\rm cen}$ almost surely.
The gradient-mismatch comparison in \eqref{eq:surrogate-gradient-error}--\eqref{eq:last-iterate-true-gap} 
then gives $\mathcal C(x_\infty^{\rm cen})-\mathcal C^\ast\leq D_2$.  
Finally, boundedness and continuity of $\mathcal C$ and dominated convergence establish the stated expected limit.

\bibliographystyle{plain}        
\bibliography{autosam}           
\end{document}